\documentclass[11pt]{article}
\usepackage{amsmath,amsthm,amsfonts,amssymb}
\usepackage{epsfig}
\usepackage[usenames]{color}
\usepackage{verbatim}

\parskip 	\bigskipamount

\newtheorem{theorem}{Theorem}[section]
\newtheorem{definition}[theorem]{Definition}

\numberwithin{equation}{section}
\newtheorem{lemma}[theorem]{Lemma}
\newtheorem{proposition}[theorem]{Proposition}
\newtheorem{corollary}[theorem]{Corollary}

\newtheorem{remark}[theorem]{Remark}

\newtheorem{claim}{Claim}[section]
\numberwithin{equation}{section}

\def\Z{\mathbb{Z}}

\def\R{\mathbb{R}}

\def\bP{\mathbb{P}}

\renewcommand{\phi}{\varphi}
\renewcommand{\epsilon}{\varepsilon}

\newcommand{\1}{{\text{\Large $\mathfrak 1$}}}

\newcommand{\bin}{\operatorname{Bin}}

\newcommand\be{\begin{equation}}
\newcommand\ee{\end{equation}}
\def\bZ{\mathbb{Z}}
\def\Reff{R_{\rm{eff}}}
\def\bP{\mathbb{P}}
\def\bE{\mathbb{E}}

\def\eps{\varepsilon}
\def\sT{\mathcal{T}}

\begin{document}

\title{\bf Collisions of Random Walks}

\author{
Martin T. Barlow\thanks{Department of mathematics, University of British Columbia, Vancouver, Canada; Research partially supported by NSERC
(Canada) and the Peter Wall Institute of Advanced Studies;   barlow@math.ubc.ca}  \and  Yuval Peres\thanks{Microsoft Research, Redmond, Washington, USA; peres@microsoft.com} \and Perla Sousi\thanks{University of Cambridge, Cambridge, UK;   p.sousi@statslab.cam.ac.uk}
}

\maketitle


\begin{abstract}
A recurrent graph $G$ has the infinite collision property if two independent random walks on $G$, started at
the same point, collide infinitely often a.s.
We give a simple criterion in terms of Green functions for
a graph to have this property, and use it to prove that a critical Galton-Watson tree with finite variance conditioned to survive, the incipient infinite cluster
in $\Z^d$ with $d \ge 19$ and the uniform spanning tree in $\Z^2$ all have the
infinite collision property. For power-law combs and spherically symmetric trees, we determine precisely the phase boundary for the infinite collision property.
\newline
MSC 2010 \emph{subject classifications.} Primary 60J10, 60J35; Secondary 60J80, 05C81.
\end{abstract}

\section{Introduction}
Let $G$ be an infinite connected recurrent graph, and let $X$ and $Y$ be independent (discrete time) simple
random walks on $G$. For classical examples such as $\bZ$ or $\bZ^2$ it is easy to see that $X$ and $Y$ collide
infinitely often -- that is $Z = |\{t: X_t = Y_t\}| = \infty$, a.s.
However, Krishnapur and Peres \cite{KP} gave an example (the graph Comb($\Z$) which is described below)
of a recurrent graph for which the number of collisions $Z$ is a.s. finite.
This had an element of surprise, as this graph is recurrent, whence the expected number of collisions is infinite, see the remarks following Theorem 1.1 of \cite{KP} .
In this paper we study the finite collision property in more detail. We start by establishing a simple zero one law
and a sufficient condition (in terms of Green functions) for infinite collisions.
Using this we show that a critical Galton-Watson tree (conditioned to
survive forever),  the incipient infinite cluster in high dimensions, and the uniform spanning tree in two dimensions
all have  the infinite collision property.

We then examine subgraphs of Comb($\Z$) and  investigate when  they have the infinite collision property, and then
conclude the paper by looking at a class of spherically symmetric trees.

We begin by defining what we mean by the finite/infinite collision property. Throughout this paper we
will only consider connected graphs.

\begin{definition}
Let $G$ be a graph, and $X$, $Y$ be independent (discrete time) simple random walks on $G$. We write
$P_{a,b}$ for the law of the process $((X_t, Y_t), t \in \Z_+)$ when $X_0=a, Y_0=b$.
Let
$$ Z = \sum_{t=0}^\infty \1( X_t = Y_t)  $$
be the total number of collisions between $X$ and $Y$. If
\be
 P_{a,a}(Z < \infty) =1 \quad \hbox{ for all } a \in G
 \ee
then $G$ has the {\bf finite collision property}.
 If
\be
 P_{a,a}(Z = \infty) =1 \quad \hbox{ for all } a \in G
 \ee
then $G$ has the {\bf infinite collision property.}
\end{definition}

We will see below that these are the only two possibilities.

We recall the definition of Comb($\Z$):

\begin{definition}
Comb($\Z$) is the graph with vertex set $\Z \times \Z$ and edge set
\[
\{ [(x,n),(x,m)]: |m-n|=1 \} \cup \{[(x,0),(y,0)]: |x-y|=1]\}
\]
\end{definition}

\begin{definition}\label{defcomb}
Following \cite{Chen}, we define the wedge comb with  profile $f$, denoted Comb($\Z, f$)
 to be the subgraph of Comb($\Z$) with vertex set
\[
V = \{(x,y) \in \Z^2: 0\leq y \leq f(x) \}
\]
and edge set the set of edges of Comb($\Z$) with vertices in $V$.
We write Comb($\Z, \alpha$)  for the wedge comb with profile $f(k)=k^\alpha$.

\end{definition}

In \cite{Chen} it is proved that Comb($\Z, \alpha$)
has the infinite collision property when $\alpha < 1/5$.

We have the following phase transition:
\begin{theorem} \label{comb}
\begin{description}
\item[(a)] If $\alpha \leq 1$, then Comb($\Z, \alpha$) has the infinite collision property.
\item[(b)] If $\alpha > 1$, then  Comb($\Z, \alpha$) has the finite collision property.
\end{description}
\end{theorem}

We remark that the proofs of both (a) and (b) extend to the profiles of the form $f(x) = C|x|^{\alpha}$.
Part (b) for $1<\alpha <2$ was also obtained independently by J. Beltran, D.Y. Chen, T. Mountford and D. Valesin (private communication).

\begin{remark} \label{no-mon}
{\rm
This theorem shows that if the `teeth' in the comb are large then the finite collision
property will hold, while it fails if they are small. However,
there is no simple monotonicity property for the finite collision property: Comb($\Z$)
has the finite collision property but is a subgraph
of $\bZ^2$, which does not.

Further, we do not have any kind of `bracketing'
property for collisions: we have
$\text{Comb}(\Z,1) \subset \text{Comb}(\Z,2) \subset \bZ^2 \subset \text{Comb}(\Z^2)$;
and of these $\text{Comb}(\Z,1)$ and $\bZ^2$ have the infinite collision property
while the other two graphs have the finite collision property.
(See \cite{KP} for the definition of $\text{Comb}(\Z^2)$, and the proof that it has the
finite collision property.)
 }
\end{remark}

In Section \ref{sec:green} we will obtain a criterion, in terms of Green functions, or equivalently
electrical resistance, for a graph to have the infinite collision property. Using this, we can show
that several graphs arising in critical phenomena have the  infinite collision property.

\begin{theorem} \label{galton}
The following random graphs all have the  infinite collision property:
\begin{description}
\item[(a)] A critical Galton Watson tree with finite variance conditioned to survive forever.
\item[(b)] The incipient infinite cluster for critical percolation in dimension $d \ge 19$.
\item[(c)] The Uniform Spanning Tree (UST) in $\Z^2$.
\end{description}
\end{theorem}

For background on the critical Galton Watson tree conditioned to survive, see \cite{Kesten}.
For background on the incipient infinite cluster and the UST, see \cite{KN} and \cite{BM}, respectively.
See Corollary \ref{cor:GWinf} for a class of critical Galton-Watson trees with infinite variance.

Another graph for which we can prove the infinite collision property is the supercritical percolation cluster in $\Z^2$, see Theorem \ref{percolation}. 
This was proved independently by Chen and Chen \cite{Chen2}.

Finally we examine some spherically symmetric trees.
\begin{definition}
A tree is called \textbf{spherically symmetric} if every vertex at distance $n$ from the root has the same number of children.
Let $(b_j)_j$ be a sequence of positive integers. We define the
\textbf{spherically symmetric tree associated to the sequence
$(b_j)_j$} as follows:  we attach a segment of length $b_0$ to the root $o$ . At the end of that segment we have a branch point with two branches, each of them having
length $b_1$, and so on.
\end{definition}

\begin{center}
\epsfig{file=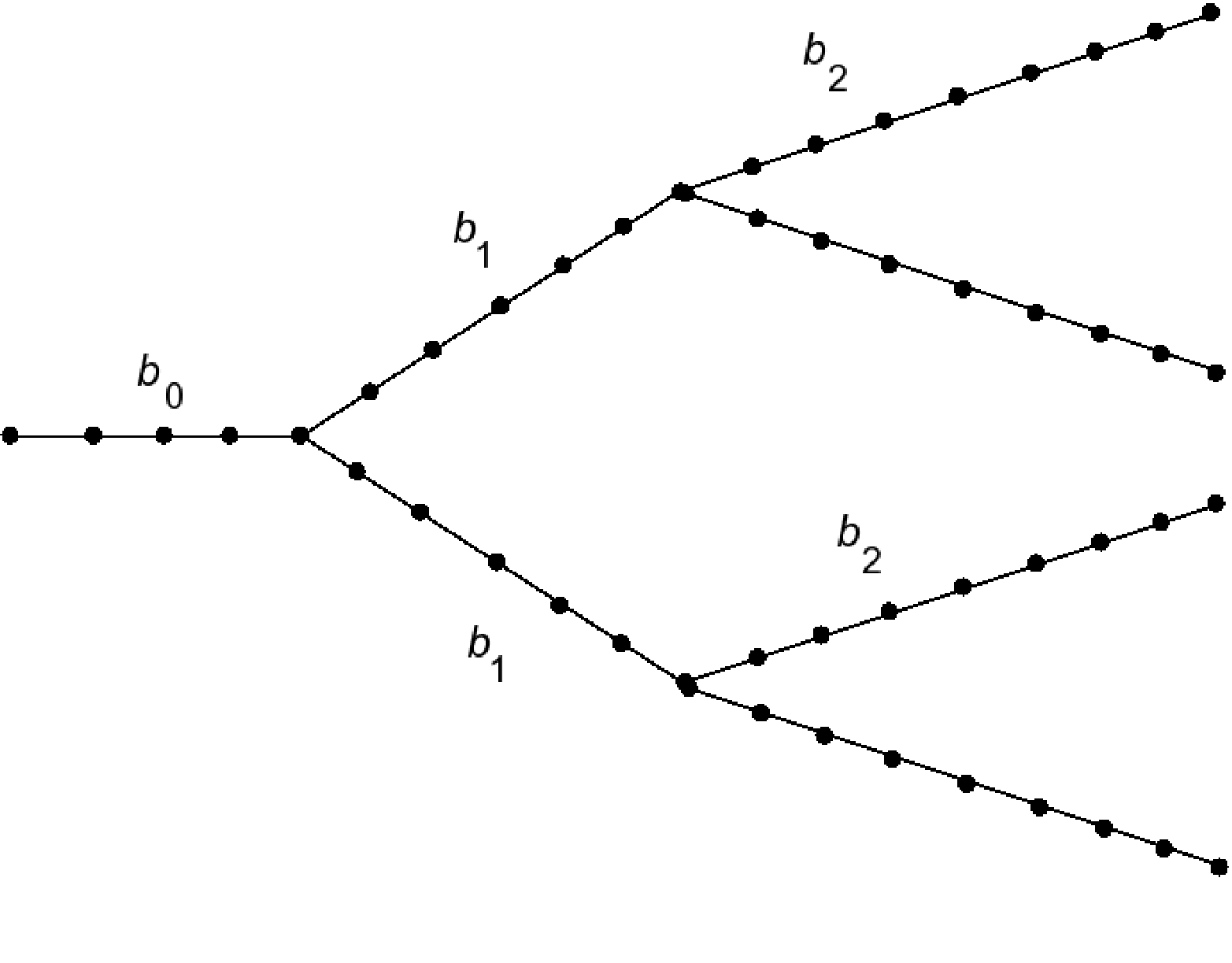,height=6cm}
\\ Figure 1. A spherically symmetric tree.
\end{center}

We will look at a class of spherically symmetric trees where the lengths are of the form $b_j=2^{2^{\beta j}}$, where $\beta>0$,
and will show that these trees exhibit two  phase transitions:
the critical parameter for recurrence of the product chain on $T \times T$ is $\beta=2$,
while the critical parameter for
the infinite collision property is $\beta=1/2$. We establish this in the following Theorem.

\begin{theorem}
\label{trees}
\begin{description}
  \item[(a)] When $\beta \geq 2$, the product chain on $T \times T$ is recurrent, and hence the tree has the infinite collision property.
  \item[(b)] When $\beta < 2$, the product chain on $T \times T$ is transient.
  \item[(c)] When $\beta \geq \frac{1}{2}$, the tree has the infinite collision property.
  \item[(d)] When $\beta < \frac{1}{2}$, the tree has the finite collision property.
\end{description}
\end{theorem}

\section{0-1 Law} \label{01law}

In this section we are going to prove that the event of having
infinitely many collisions in a recurrent graph is a trivial event,
and hence has probability either 0 or 1.  Thus in order to show the
infinite collision property it suffices to show that infinitely many
collisions occur with positive probability.

\begin{proposition} \label{0-1}
Let $G$ be a (connected) recurrent graph, $X$ and $Y$ be independent random walks on $G$, and $Z$ be the number of
collisions. Then for each $(a,b) \in G \times G$,
\begin{align*}
P_{a,b}(Z=\infty) \in \{ 0,1 \} .
\end{align*}
Further, if there exist $a_0$, $b_0$ such that $P_{a_0, b_0}(Z=\infty)>0$ then
$P_{a,b}(Z=\infty)=1$ for all $a$, $b$ such that $P_{a,b}(X_m=a_0, Y_m=b_0)>0$ for some $m \ge 0$.
In particular, either $P_{a,a}(Z=\infty)=0$ for all $a$ or else  $P_{a,a}(Z=\infty)=1$ for all $a$.
\end{proposition}

\proof

Let $\sT^X_n=\sigma( X_n, X_{n+1}, \dots)$, and define $\sT^Y_n$ analogously.
Then since $X$ is a recurrent Markov chain $\sT^X = \cap_n \sT^X_n$ is trivial by Orey's theorem (see \cite{Orey}).
By \cite[Lemma 2]{LR} we have, since $X$ and $Y$ are independent, that
$$ \sT= \bigcap_{n=1}^\infty \sigma( \sT^X_n, \sT^Y_n) = \sigma(\sT^X, \sT^Y), $$
which is trivial since $\sT^X$ and $\sT^Y$ are both trivial.
Since the event $\{ Y_n=X_n \text{ i.o.}\}$ is $\sT$ measurable, it therefore has probability
0 or 1.

Now suppose $P_{a_0, b_0}(Z=\infty)=1$ and let $a,b,m$ be as above, i.e. $P_{a,b}( X_m =a_0, Y_m =b_0)>0$.
Then
\[
P_{a,b}(Z = \infty) \ge  P_{a,b}(Z = \infty | X_{m}=a_0, Y_{m}=b_0) P_{a,b}( X_m = a_0, Y_m =b_0)>0.
\]
By the 0-1 law therefore $P_{a,b}(Z = \infty)=1$.
\qed


\begin{remark} \label{zeroone}
{\rm The proof of Proposition \ref{0-1} applies to any recurrent chain.
Note that if $Z'$ denotes the total number of edges that are crossed at the same time by
two independent random walks on a recurrent graph
(started from the same state), then the event $\{Z' =\infty \}$ has probability zero or
one (since the sequence of edges crossed by a recurrent random walk forms a recurrent chain.)
}
\end{remark}

\begin{corollary} \label{cor:fsets}
Let $A_n$ be finite subsets of $G$, let
$$ Z(A_n) := \sum_t  \1(X_t=Y_t \in A_n)$$
be the number of collisions in $A_n$, and $F_n =\{ Z(A_n)>0\}$.  \\
(a) If  $G = \cup_n A_n$ and
$P( F_n \hbox { occurs i.o.} )=0$ then $G$ has the finite collision property. \\
(b) If $A_n$ are disjoint and  $P(F_n) > c >0$ for all $n$ then  $G$ has the infinite collision property.
\end{corollary}

\begin{proof} (a) If $G \times G$ is recurrent then there are a.s. infinitely many collisions at each
point $x \in G$, and so  $P( F_n \hbox { occurs } i.o. )=1$. We can therefore assume that
$G \times G$ is transient. Hence there are only finitely many collisions in each set $A_n$,
and as the total number of sets $A_n$ with a collision is finite, the total number of collisions
is finite. \\
(b) We have $P( F_n  \hbox { occurs i.o.} )> c$. However,  $Z \ge \sum_n \1_{F_n}$, and
so $P(Z = \infty)>c$. So by the 0-1 law, Proposition \ref{0-1}, we get $P(Z=\infty)=1$.
\end{proof}

\section{Green function criterion for $\infty$ collisions} \label{sec:green}
\subsection{Background material}

Firstly we are going to recall a few facts about heat kernels and effective resistances.
We will follow rather closely the exposition   in \cite{Barlow3} and \cite{Barlow}.
Let $d(x)$ denote the degree of a vertex $x$ in a graph $G$. For two functions
$f$, $g$ $\in \R^{V(G)}$ we define the quadratic form
\begin{align*}
\mathcal{E}(f,g) = \frac{1}{2} \sum_{\substack{x,y \in V(G) \\ x\sim y}}(f(x)-f(y))(g(x)-g(y)).
\end{align*}

We define the transition density
\[
q_t(x,y) = \frac{P_x(Y_t = y)}{d(y)}, \quad t \in \bZ_+.
\]
Here we have divided by the degree of the vertex to make the transition density a symmetric function.

Let $A$ and $B$ be two subsets of $V(G)$.
The effective resistance between $A$ and $B$ is defined as follows:
\begin{align}
\label{effective}
R_{\text{eff}}(A,B)^{-1} = \inf \{ \mathcal{E}(f,f): \mathcal{E}(f,f)< \infty, f|_{A} = 1, f|_{B} = 0\}.
\end{align}
The term effective resistance comes from electrical network theory, since we can think of
our graph as an electrical network having unit resistances wherever there is an edge between two vertices.
If we glue all points of $A$ to a point $a$  and all points of $B$ to $b$ and apply a voltage $V$
which then induces a current $I$ from $a$ to $b$, then the ratio $\frac{V}{I}$ is constant and is equal to the effective
resistance.

From the definition \eqref{effective} of effective resistance we see that there is a unique function $f$ achieving
the infimum appearing on the right hand side of \eqref{effective}.
This function must be harmonic everywhere outside the sets $A$ and $B$.

For any graph $G$ the effective resistance satisfies $R_{\text{eff}}(x,y)\leq d(x,y)$ and if
 $G$ is a tree, then
\[
R_{\text{eff}}(x,y) = d(x,y),
\]
where $d(x,y)$ stands for the graph theoretic distance between $x$ and $y$.

Let $B(x_0,r) = \{y: d(x_0,y) \leq r \}$ and $Y_t^B$ ($B:= B(x_0,r)$) be the discrete time
simple random walk on $G$ killed when it exits $B(x_0,r)$ and let $q_t^B$ be its transition density.
The Green kernel is defined  by $g_B(x,y) = \sum_{t=0}^{\infty} q_t^B(x,y)$.

It is easy to see that $g_B(x,\cdot)$ is a harmonic function on $B\setminus \{x \}$ and that it satisfies the
reproducing property, i.e. that $\mathcal{E} (g_B(x,\cdot),f) = f(x)$, for any function $f$ satisfying $f|_{B^c}=0$.

The function defined by $h(y) := \frac{g_B(x,y)}{g_B(x,x)}$ is harmonic on $B\setminus \{x \}$ and takes
value 1 at $x$ and 0 on $B^c$, hence $R_{\text{eff}}(x,B^c)^{-1} =\mathcal{E}(h,h)$.
Using now the reproducing property mentioned above we get that
\[
R_{\text{eff}}(x,B^c) = g_B(x,x),
\]
a very useful equality that will be widely used in this paper.

If $B$ is a finite subset of $G$ then by spectral theory we can write
\begin{align}
\label{spectral}
q^B_t(x,y ) = \sum_i \lambda_i^t \phi_i(x) \phi_i(y) ,
\end{align}
where $\phi_i$ are the eigenfunctions and $\lambda_i$ the eigenvalues of the associated transition operator.
Since $| \lambda_i| \le 1$ for all $i$ it follows that
\be \label{e:qmon}
   q^B_{2t+1}(x,x) \le q^B_{2t}(x,x) \quad \hbox{ for all } x \in B, \,  t \ge 0.
\ee
Letting $B \uparrow G$ this inequality extends to $q_t$. From \eqref{e:qmon} we obtain
\be \label{e:2gB}
 2 g_B(x,x) \ge 2 \sum_{t=0}^\infty q_{2t}(x,x) \ge  \sum_{t=0}^\infty ( q_{2t}(x,x) +  q_{2t+1}(x,x))
 = g_B(x,x).
 \ee

\subsection{The criterion}
\begin{theorem}
\label{greenkernel}
Let $G$ be a recurrent graph with a distinguished vertex $o$. Let  $(B_r)_r$ be
an increasing  sequence of  sets such that $B_r \neq G, \forall r$, and
$\cup_r B_r = G$. Suppose that there exists $C < \infty$ such that  for all $r$
\begin{align*}
g_{B_r}(x,x) \leq C g_{B_r}(o,o), \quad \text{ for all  } x \in B_r,
\end{align*}
Then $G$ has the infinite collision property. Moreover, the number of edges crossed at the same time by two independent random walks is infinite a.s.
\end{theorem}

\proof
Let $(B_r)_r$ be the sequence of sets satisfying the assumptions of the theorem.
Set $B:=B_r$ and let $X^B$ and $Y^B$ be the two random walks killed after exiting the
set $B$, and let $q^B_t$ be their transition densities.
Let $\widetilde Z_B$ count the number of edges that are crossed at the same time by these two random walks, i.e.
\[
\widetilde Z_B = \sum_{t=0}^\infty \1( X^B_t = Y^B_t, X_{t+1} = Y_{t+1}).
\]

To prove the theorem we are going to apply the second moment method to the random variable $\widetilde Z_B$,
so we begin by computing its first and second moments.
For the first moment we have
\begin{align*}
E_{o,o} (\widetilde Z_B)  &= \sum_{t}\sum_{x \in B} \sum_{y\sim x}
     P_{o,o}(X^B_t=Y^B_t=x, X_{t+1} = Y_{t+1} =y)    \\
    &= \sum_{t}\sum_{x \in B} \sum_{y\sim x} q^B_t(o,x)^2 d(x)^2 q_1(x,y)^2 d(y)^2 \\
    &= \sum_{t}\sum_{x \in B} q^B_t(o,x)^2 d(x) =\sum_{t=0}^\infty q^B_{2t}(o,o).
\end{align*}
We therefore have
\be \label{e:zr}
     g_{B}(o,o) \ge    E_{o,o} (\widetilde Z_B)  \ge  \frac12 g_{B}(o,o).
\ee
Observe that since $B_r \neq G$ and $G$ was assumed to be a recurrent graph, we have that $g_{B_r}(o,o)<\infty$.

And for the second moment we have
\begin{align} \nonumber
E_{o,o}(\widetilde Z_B^2) &= E_{o,o}(\widetilde Z_B) +  2 \sum_{t}\sum_{s\geq t+1} \sum_{x \in B} \sum_{y \sim x }
\sum_{z \in B} \sum_{w \sim z }  q^B_t(o,x)^2 d(x)^2 q_1(x,y)^2 d(y)^2  \\
  \nonumber
 & \qquad \qquad \times  q^B_{s-t-1}(y,z)^2 d(z)^2 q_1(z,w)^2 d(w)^2  \\
  \nonumber
 &=   E_{o,o}(\widetilde Z_B) + 2 \sum_{t}\sum_{s\geq t+1} \sum_{x \in B} \sum_{y \sim x }
\sum_{z \in B}  q^B_t(o,x)^2 q^B_{s-t-1}(y,z)^2 d(z)  \\
 \nonumber
&\le    E_{o,o}(\widetilde Z_B) + 2 \sum_{t} \sum_{x \in B} \sum_{y \sim x }  q^B_t(o,x)^2  g_B(y,y) \\
 \label{e:second}
&\le   g_B(o,o)  + 2 g_B(o,o) \max_{y \in B}  g_B(y,y).
\end{align}


Applying the second moment method to the variable $\widetilde Z_{B_r}$,
and using \eqref{e:zr}, \eqref{e:second} and the hypotheses of the theorem we obtain
\begin{align*}
P_{o,o}( \widetilde Z_{B_r} > \frac{1}{2}E_{o,o}( \widetilde Z_{B_r}))
&\geq \frac{1}{4} \frac {(E_{o,o}( \widetilde  Z_{B_r} ))^2 }{E_{o,o}( \widetilde Z_{B_r}^2) }
\ge \frac{  g_B(o,o)} {16 (1 + 2C g_B(o,o))}.
\end{align*}
Since $g_{B_r}(o,o) \ge d(0)^{-1}$, it follows that
$P_{o,o}(\widetilde Z_{B_r} > \frac{1}{4}g_{B_r}(o,o)) \geq c >0$, for all $r>0$.
As $r \to \infty$ we have  $\widetilde Z_{B_r} \nearrow \widetilde Z$,
where $\widetilde  Z$ is the total number of common edges
traversed by $X$ and $Y$.
Letting $r \to \infty$ , we get $P_{o,o}(\widetilde Z = \infty) >c$.
Since $Z \ge \widetilde  Z$, we have $P_{o,o}(Z = \infty) >c$, and
so by the 0-1 law, Proposition \ref{0-1}, we get  $P_{o,o}(Z = \infty) =1$. For the last statement use Remark \ref{zeroone}.
\qed

The proof of \eqref{e:zr} also gives

\begin{lemma} \label{lem:GZB}
Suppose that $d(x) \le D$ for all $x \in G$. Let $Z_B$ be the total number of
collisions of the killed walks $X^B$ and $Y^B$. Then
$$ \frac12 g_B(o,o) \le E_{o,o} Z_B \le D g_B(o,o). $$
\end{lemma}

\subsection{Applications of the Green kernel criterion}

We now give a number of applications of this criterion, and in particular will prove
Theorem \ref{comb}(a) and Theorem \ref{galton}.

\begin{proof}[\textsl{Proof of Theorem \ref{comb} (part \textbf{a})}]
Let $B:=B_r$ denote the set of vertices that are on the right of the origin and at horizontal distance at most $r$ from it
-- see Figure 2 below.
Then $g_{B_r}(0,0) = R_{\text{eff}}(0,B_r^c) = d(0,B_r^c) = r+1$ and since $\alpha \leq 1$ we have that $g_{B_r}(x,x) = R_{\text{eff}}(x,B_r^c) = d(x,B_r^c) \leq r + 1 = g_{B_r}(0,0)$, for any $x \in B_r$.
\end{proof}

\begin{center}
\epsfig{file=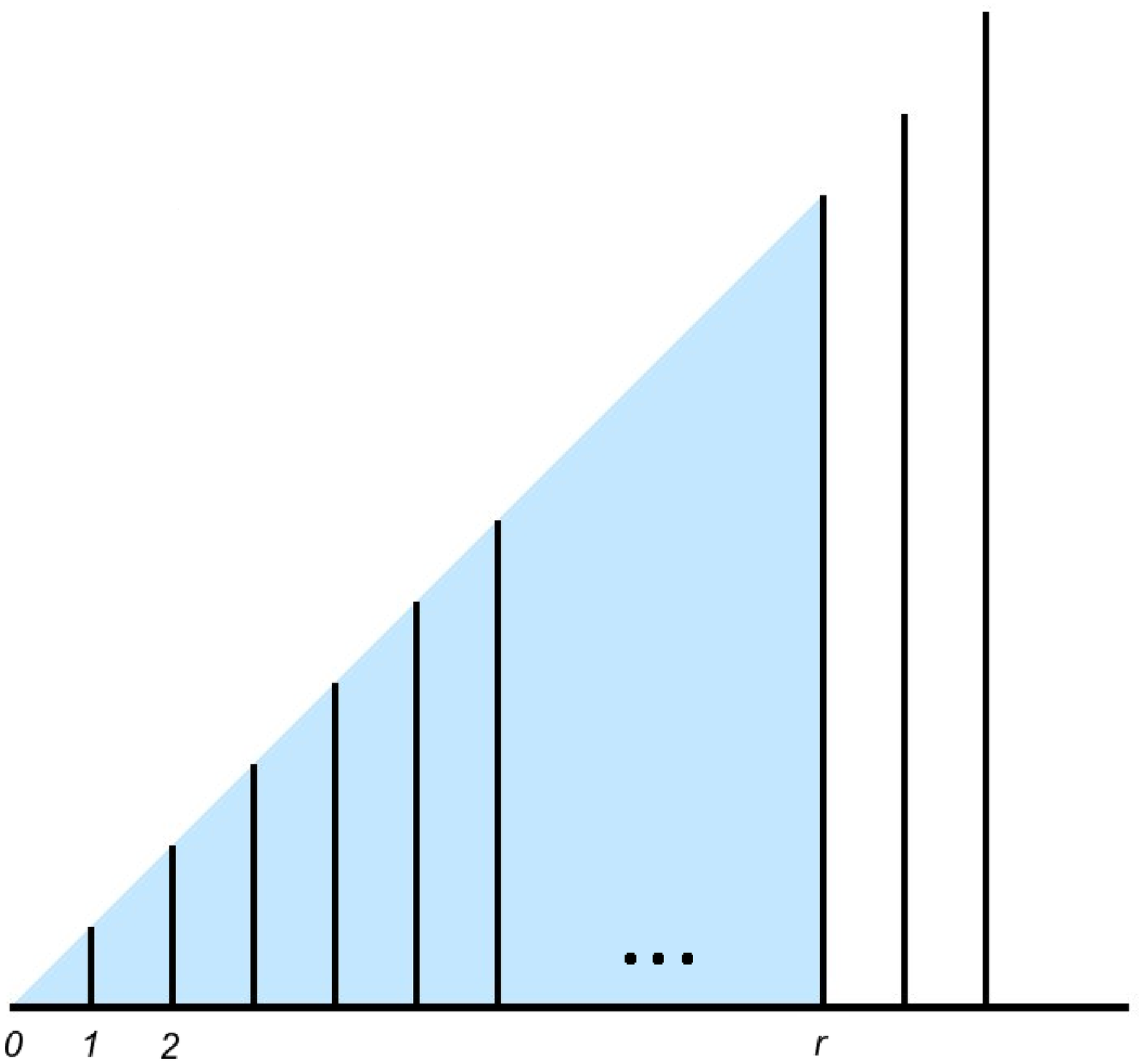,height=6cm}
\\ Figure 2. The set $B_r$ in the wedge comb.
\end{center}

\begin{proof}[\textsl{Proof of Theorem \ref{galton}}(a)]
In this proof we have two types of randomness. We define the Galton-Watson tree
on a probability space $(\Omega, \bP)$, and denote the tree $G(\omega)$, and its root
$o$.  We then write $P^x_\omega$ for the law of the simple random walk in $G(\omega)$
started at $x \in G(\omega)$.

Let us quickly recall the structure of the critical Galton Watson tree with finite variance conditioned
to survive forever. For more details see for instance \cite{Kesten}.
Let $(p_k)$ be the offspring distribution of the critical Galton
Watson tree. Now start with the root $o$. Give it a random number of
offspring which follows the size-biased distribution, i.e.  $P(X=k) =
k p_k$. The random variable $X$ has finite expectation, since the
original distribution $p_k$ has finite variance.  Choose one of its
offspring at random and give it a random number of offspring with the
size-biased distribution independently of before, and to all the others
attach critical Galton Watson trees with the same offspring
distribution $(p_k)$.

From this construction it follows that there is a unique infinite
line of descent, which we call the {\bf backbone} and off the nodes on it
there are critical finite trees emanating.

Let $B_r$ be the set of vertices on the backbone that are at distance at most $r$ from the root,
taken together with all their descendants that are off the backbone -- see Figure 3.

\begin{center}
\epsfig{file=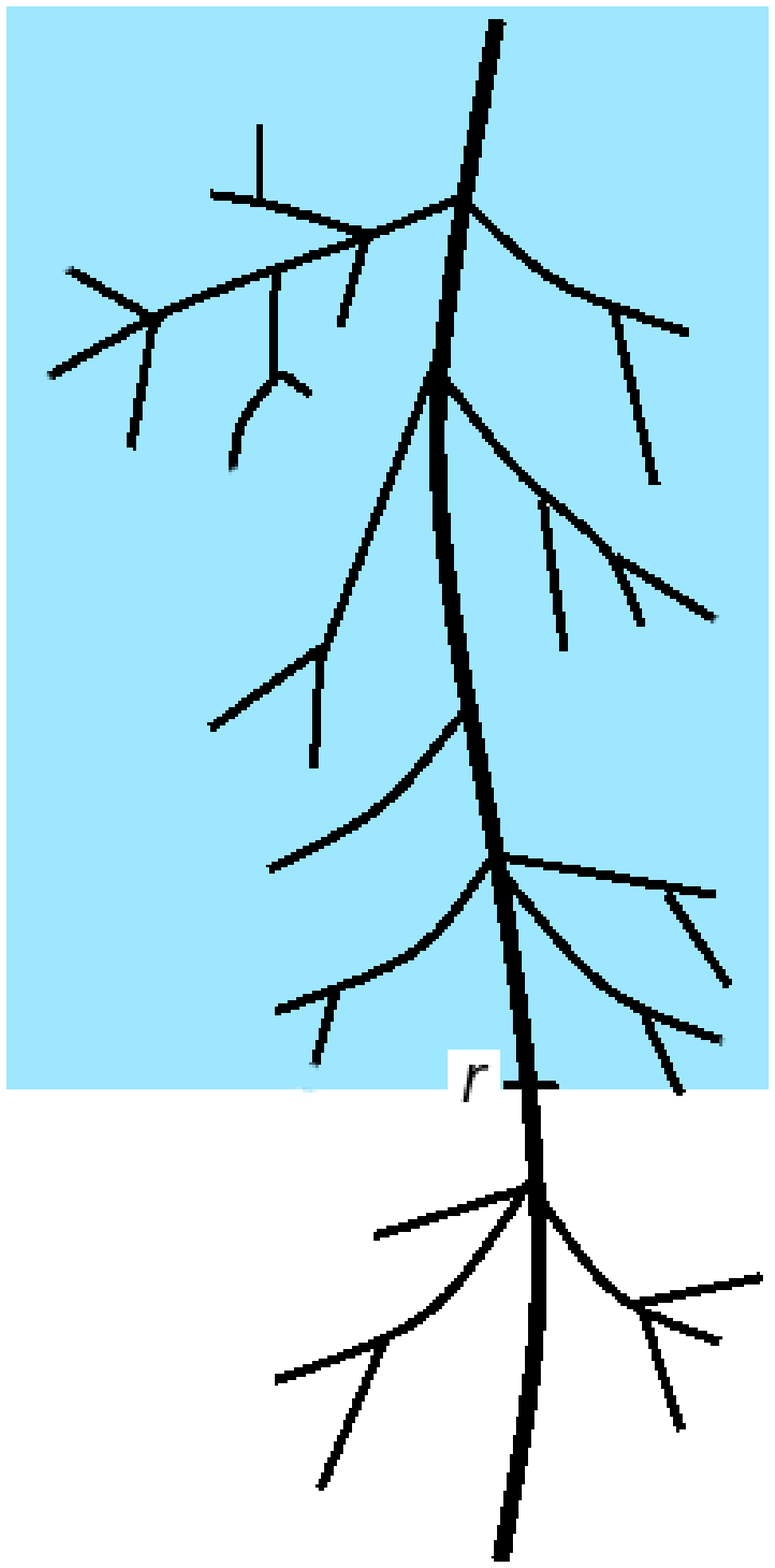,height=6cm}
\\ Figure 3. A Galton-Watson tree with the set $B_r$.
\end{center}

Fix $\epsilon>0$.
Let $N^{\epsilon}_r$ be the number of trees of depth greater than $\frac{r}{\epsilon}$ that are contained in the set $B_r$,
excluding the backbone itself.
If $(Z_n)$ is a critical branching process with finite variance, then Kolmogorov's theorem states that
\begin{align}
\label{Kolmogorov}
P(Z_n>0) \sim \frac{2}{n\sigma^2} \text{ as } n\to \infty.
\end{align}

Let $Y_i$, for $i=0,\cdots,r$, be the number of offspring of the $i$-th vertex on
the backbone excluding the offspring on the backbone.
Then $\bE(Y_i) = \sum_{k=1}^{\infty}k^2p_k  - 1= \sigma^2$.
We label the offspring of the $i-$th vertex on the backbone by $j=1,\cdots, Y_i$ if $Y_i \geq 1$.
Also, we let $T_{i,j}$, for $j=1,\cdots,Y_i$, be the descendant tree of the $j$-th child off
the backbone.
Using \eqref{Kolmogorov} we have
\begin{align*}
\bP(N^{\epsilon}_r \geq 1) \leq \bE(N^{\epsilon}_r)
= \bE\left( \sum_{i=0}^{r} \sum_{j=1}^{Y_i} \1\left(T_{i,j} \text{ has depth }> \frac{r}{\epsilon}\right)\right)
\leq \sum_{i=1}^{r} \bE(Y_i) \left( \frac{c \epsilon}{\sigma^2 r} \right)
\leq C \epsilon.
\end{align*}
So $\bP(N^{\epsilon}_r=0) \geq 1-C\epsilon$ and by Fatou's lemma we have, setting
 $A_\epsilon = \{\omega: N^{\epsilon}_r(\omega)=0 \text{ i.o.} \}$, that
\begin{align*}
\bP(A_\epsilon) =
\bP(N^{\epsilon}_r = 0 \text{ i.o.}) \geq \lim\sup_{r} \bP(N^{\epsilon}_r=0)  \geq 1-C\epsilon.
\end{align*}
%
%
Now $g_{B_r}(o,o) = r+1$, and if $N^{\epsilon}_r=0$ then $g_{B_r}(x,x) \le r + r/\epsilon$
for all $ x \in B_r$.
If $ \omega \in A_\epsilon$ then
applying the Green kernel criterion to the sets $B_r$ with $r$
being such that $N^{\epsilon}_r(\omega)=0$, we get the infinite collision property for the graph $G(\omega)$.
Hence we deduce that
\[ \bP( G(\omega) \text{ has the infinite collision property } ) \geq P(A_\epsilon) \geq 1-C\epsilon, \]
and thus sending $\epsilon \to 0$, we get that
$G$  has the infinite collision property $\bP$-a.s.
\end{proof}


We have the following easy corollary of Theorem \ref{greenkernel}

\begin{corollary} \label{cor:J}
Let $(G(\omega))$ be a family of random graphs (defined on a space $(\Omega, \bP)$)
with a distinguished vertex $o$, and let $B_r = B(o,r)$. For $\lambda\ge 1$ let
$$ J(\lambda) = \{ r \in \Z_+:  \Reff (o, B_r^c) \ge r/ \lambda \}. $$
Suppose that there exists a function $\psi(\lambda)$ with $\lim_{\lambda \to \infty} \psi(\lambda)=0$
and $r_0 \ge 1$ such that
\be \label{e:Jlam}
 \bP( r \in J(\lambda) ) \ge 1 -  \psi(\lambda) \quad \hbox{ for all } r \ge r_0.
 \ee
 Then $G$ has the infinite collision property $\bP$-a.s.
\end{corollary}

\begin{proof}
For each $x \in B_r$ we have
$g_{B_r}(x,x) \le d(x,o) + r \le 2r$,
while for each $r \in J(\lambda)$
$$ g_{B_r}(o,o) = \Reff(o, B_r^c) \ge r/\lambda. $$
The condition \eqref{e:Jlam} implies that
$$ \bP( r \in  J(\lambda) \mbox{ \rm for infinitely many }   r ) \ge 1-\psi(\lambda) . $$
If this event holds then Theorem \ref{greenkernel} implies that
$G$ has the infinite collision property.
Letting $\lambda \to \infty$ concludes the proof.
\end{proof}

\begin{proof}[\textsl{Proof of Theorem \ref{galton}} (b) and (c)]
For both of these graphs the condition \eqref{e:Jlam} has been verified. For the
incipient infinite cluster in dimension $d \ge 19$ see the proof of (2.1) at the end of
section 2 of \cite{KN}. For the UST see  Proposition 3.6 of \cite{BM}.
\end{proof}

\medskip
\begin{remark} {\rm
We could also have used Corollary \ref{cor:J} to prove Theorem \ref{galton}(a), since
\cite[Proposition 1.1]{FK} proves that a critical Galton-Watson tree with finite variance
conditioned to survive forever
satisfies the condition \eqref{e:Jlam}. However, we preferred to give a
simple direct proof.
} \end{remark}

We can also use Corollary
 \ref{cor:J} to handle a class of critical Galton-Watson trees with infinite variance.

\begin{corollary} \label{cor:GWinf}
Let $(Z_n)$ be a critical Galton-Watson process with infinite variance such that
$$ E( s^{Z_1}) = s + (1-s)^\alpha L(1-s), $$
where $\alpha \in (1,2]$, and $L(t)$ is slowly varying as $t \to 0$.
Let $T^*$ be the tree associated with the process $(Z_n)$ conditioned to survive forever.
Then $T^*$ has the infinite collision property.
\end{corollary}

\begin{proof} The condition \eqref{e:Jlam} for this tree is proved in
\cite[Lemma 3.1]{CK}.
\end{proof}

We also have that many `fractal' graphs satisfy the infinite collision property. Examples of
graphs of this kind are given in Examples 3 and 4 in Section 5 of \cite{Barlow3}: these include the graphical
Sierpinski gasket -- see Figure 1 in \cite{KN}.
All these graphs have bounded vertex degree, and there exist $\beta\ge 2$ and
$\alpha \in [1, \beta)$ such that for $x, y \in G$, $r \ge 1$
$$ |B(x,r)| \asymp r^\alpha, \qquad \Reff(x,y) \asymp d(x,y)^{\beta-\alpha}. $$
Lemma 2.2 of \cite{Barlow3} then proves that
$$ \Reff(x, B(x,r)^c) \ge c r^{\beta-\alpha}. $$
We therefore have
$$\max_{y \in B(x,r)} \frac{ g_{B(x,r)}(y, y)}{  g_{B(x,r)}(x, x)}= \max_{y \in B(x,r)} \frac{ \Reff(y, B(x,r)^c)}{  \Reff(x, B(x,r)^c)} \le C $$
for all $x \in G$, $r \ge 1$. Hence the hypotheses of Theorem \ref{greenkernel} hold,
and the graph has the infinite collision property.

\section{Supercritical percolation cluster} \label{supercrit}

In this section we are going to give a short proof of the following theorem, which was proved independently in
\cite{Chen2}.

\begin{theorem}
\label{percolation}
Let $X$ and $Y$ be two independent discrete time simple random walks on the infinite supercritical
percolation cluster in $\Z^2$ started from the same point. Then $X$ and $Y$ will collide infinitely many times a.s.
\end{theorem}

\proof
We use the upper and lower bounds on the transition
density $q_t(x, y)$ as given in \cite{Barlow2}, Theorem 1. The bounds in \cite{Barlow2}
are proven for a continuous time simple random walk, but it is remarked there that the same
bounds also hold for the discrete time walk. While no details are given in \cite{Barlow2},
Section 2 of \cite{BarlowHambly} outlines the changes needed to run the proofs of
\cite{Barlow2} in discrete time.
For convenience we state these bounds here:

\begin{theorem}
\label{bounds}
Let $p > \frac{1}{2}$. There exists $\Omega_1$ with $P_p(\Omega_1) = 1$ and r.v.
$S_x,x \in \Z^2$, such that $S_x(\omega) <\infty$ for each $\omega \in \Omega_1$, $x \in C_{\infty}(\omega)$.
There exist constants $c_i = c_i(d, p)$ such that for $x, y \in C_{\infty}(\omega),t \geq 1$ with
\[
S_x(\omega)\vee |x - y|_1 \leq t,
\]
the transition density $q^{\omega}_t(x, y)$ of $Y$ satisfies
\be  \label{e:scc_est}
c_1 t^{-1} \exp(-c_2|x -y|^2_1
/t) \leq  q^{\omega}_t(x, y)\leq c_3 t^{-1} \exp(-c_4|x -y|^2_1/t).
\ee
\end{theorem}

We fix a configuration $\omega$. We are going to work with the probability measure $P_{\omega}$.
Let $ a \in C_{\infty}(\omega)$.  We define
\[
Z_k =  \sum_{t=k_0}^{k}\sum_{|x-a|_1<\sqrt{t}} \1(x,t) \1(S_x(\omega) \leq B),
\]
where the indicator $\1(x,t)$ means that there is a collision at time $t$ at position $x$, i.e. that $X_t=Y_t=x$.
In order to use the bounds given in Theorem \ref{bounds} we set $k_0 = S_a \vee 2(|a|)^2$.
As in the previous section we will use the second moment method, so we need bounds on
the first and second moments of $Z_k$. We have
\begin{align} \nonumber
E^{\omega}_{a,a}(Z_k) &=
 E^{\omega}_{a,a}\left(\sum_{t=k_0}^{k}\sum_{|x-a|_1<\sqrt{t}} \1(x,t) \1(S_x(\omega) \leq B) \right) \\
  \label{mesitimi}
&\geq \sum_{t=k_0}^{k}\sum_{|x-a|_1<\sqrt{t}, S_x(\omega) \leq B} q^\omega_t(a,x) q^\omega_t(a,x).
\end{align}
By the ergodic theorem we have that
\begin{align*}
\frac{\# \{x \in C_{\infty}: |x-a|_1 \leq \sqrt{t}, S_x \geq B \}}{4t} \to P_p(x \in C_{\infty}, S_x \geq B), \text{ as } t \to \infty,\text{ a.s. }
\end{align*}

So $\forall \omega \in \Omega_2$ with $P(\Omega_2)=1$ we have that
\[
\frac{\# \{x \in C_{\infty}: |x-a|_1 \leq \sqrt{t}, S_x(\omega) \geq B \}}{4t} \leq \delta_B <1, \text{ ev. },
\]
where the strict inequality follows from the fact that the probability $P_p(x \in C_{\infty}, S_x \geq B) \to 0$ as $B \to \infty$.
We now fix a $B$ with $B \ge S_a$ such that $\delta_B <1$.
We now use the estimate \eqref{e:scc_est} to bound from below the
transition density in \eqref{mesitimi},  and deduce that
\begin{align*}
E^{\omega}_{a,a}(Z_k) \geq \sum_{t=k_0}^{k} c \frac{1}{t} > c' \log{k}.
\end{align*}
We now bound the second moment of $Z_k$.
Note first that by \eqref{e:scc_est}
\begin{align*}
  \sum_{s=0}^{k} \sum_{|y-a|_1<\sqrt{t}} q^{\omega}_{s}(x,y)^2   &\le   \sum_{s=0}^{k} \sum_y q^{\omega}_{s}(x,y)^2 d(y) \\
  &\le  \sum_{s=0}^{k} q^{\omega}_{2s}(x,x)
  \le S_x +  \sum_{s=S_x}^{k} c s^{-1} \le S_x + c \log k.
\end{align*}

Then
\begin{align*}
E^{\omega}_{a,a}(Z_k^2) &\leq 2 E^{\omega}_{a,a}\left( \sum_{t=k_0}^{k}
 \sum_{l=t}^{k}\sum_{|x-a|_1<\sqrt{t}}\sum_{|y-a|_1<\sqrt{l}} \1(x,t) \1(y,l)
    \1(S_x(\omega) \leq B)\1(S_y(\omega) \leq B)\right)  \\
& \le c \sum_{t=k_0}^{k}\sum_{l=t}^{k}\sum_{|x-a|_1<\sqrt{t}}\sum_{|y-a|_1<\sqrt{l}}
   \1(S_x(\omega) \leq B)\1(S_y(\omega) \leq B) q_t^{\omega}(a,x)^2 q_{l-t}^{\omega}(x,y)^2 \\
&\leq c \sum_{t=k_0}^{k}\sum_{|x-a|_1<\sqrt{t}} \1(S_x(\omega) \leq B)
 q^{\omega}_t(a,x)^2 \sum_{l=t}^{k}\sum_{|y-a|_1<\sqrt{l}} q^{\omega}_{l-t}(x,y)^2 \\
 &\leq c \sum_{t=k_0}^{k}\sum_{|x-a|_1<\sqrt{t}}    q^{\omega}_t(a,x)^2  \1(S_x(\omega) \leq B) ( S_x + c \log k )
\le (B + c \log k)^2.
\end{align*}
Applying the second moment method to the random variable $Z_k$ we obtain that
\[
P^{\omega}_{a,a}(Z_k > \frac{1}{2}c \log{k}) \geq \frac{1}{4}\frac{(E^{\omega}_{a,a}(Z_k))^2}
  {E^{\omega}_{a,a}(Z^2_k)} \geq c>0.
\]
Taking limits in the above inequality as $k \to \infty$, we obtain that $P^{\omega}_{a,a}(Z = \infty) >c$,
and so by Proposition \ref{0-1}  we have $P_{a,a}(Z=\infty) =1$.
\qed

\section{Wedge comb with $\alpha >1$} \label{sec:wedge}

In Section \ref{sec:green}
we proved that wedge combs with profile $f(x)=x^\alpha$ where
$\alpha \le 1$ have the infinite collision property. In this section we will prove
Theorem \ref{comb}(b), that is that if $\alpha >1$ then the wedge comb has the finite collision property.
We do not have any simple general criterion for the finite collision property, and our proofs will
rely on making sufficiently accurate estimates of the transition density $q_t(x,y)$.

For $x \in G$ we write $x_1$ for the first coordinate of $X$.

Throughout this section we set
$$ \alpha' = \alpha \wedge 2, \quad \beta'= \frac{1+\alpha'}{2+\alpha'}. $$
Note that if $\alpha \ge 2$ then $\alpha'=2$, $\beta'=3/4$.

The main work in this section will be in proving the following.

\begin{lemma} \label{lem:qtbound}
Let $x = (k,h) \in G$. The transition density $q$ satisfies:
\begin{equation}
 q_t(0, x) \leq
\begin{cases}
\frac{c}{t^{\beta'}} \quad   & \hbox{ if }  t \ge k^{2 +\alpha'} , \\
\frac{c}{ (k^{2 + \alpha'})^{\beta'}}  \quad  & \hbox{ if }  t < k^{2 +\alpha'} .
\end{cases}
\end{equation}
\end{lemma}

Before we prove this Lemma, we will show how it leads easily to Theorem \ref{comb}(b).

We define the set $Q_{k,h}$, where $h \leq k^{\alpha}$, as follows:
\[
Q_{k,h} = \{ (k,y): 0\leq y \leq h  \},
\]
and we set $Z_{k,h}= Z(Q_{k,h})$ to be the number of collisions of the two random walks
in $Q_{k,h}$.
We also define $\tilde{Z}_{k,h} = Z_{k,2h/3} - Z_{k,h/3}$, i.e. the number of collisions that happen in the
set $\{(k,y): \frac{h}{3} \leq y \leq \frac{2h}{3}\}$.
\begin{center}
\epsfig{file=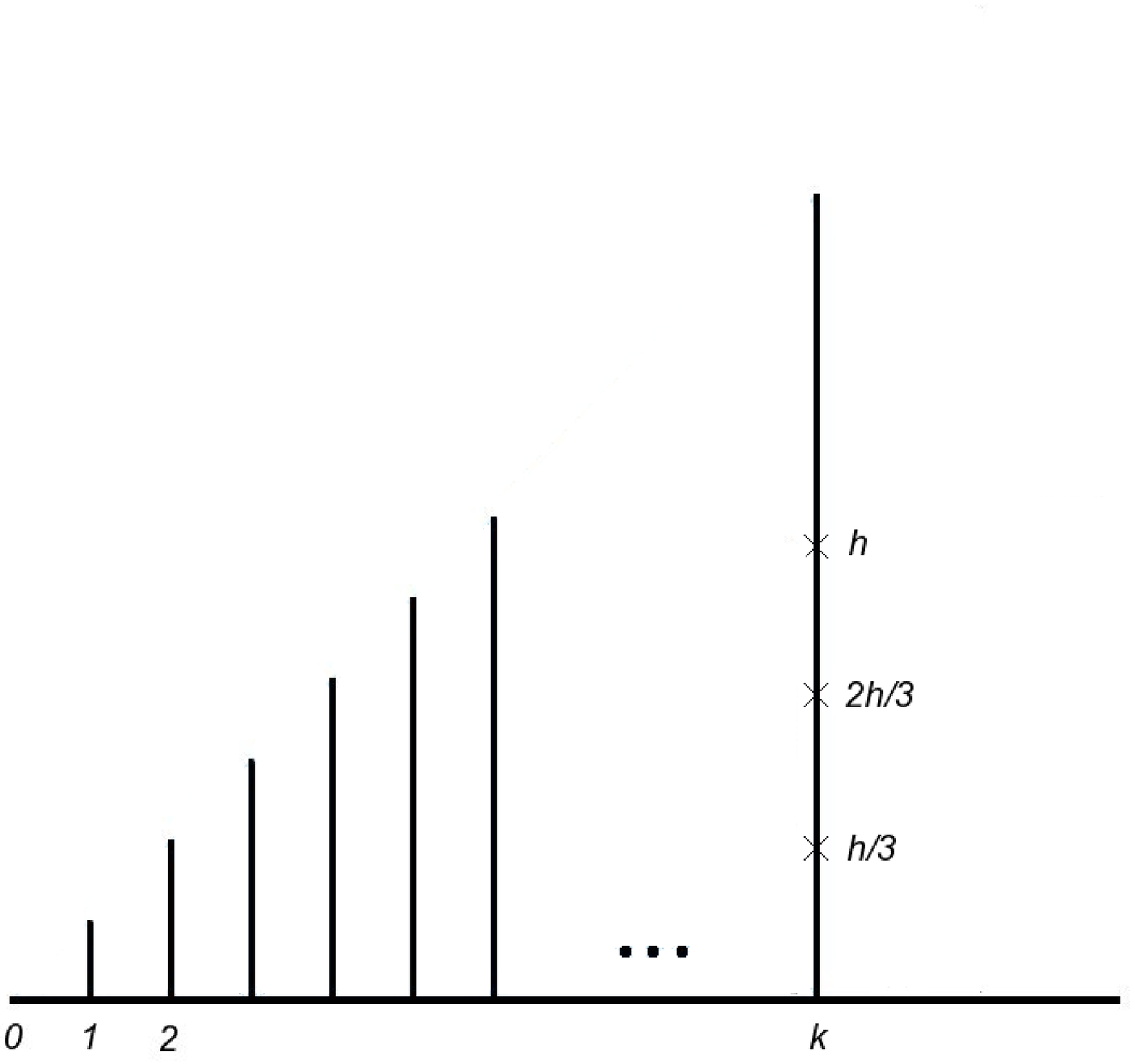,height=5cm}
\\ Fig 4
\end{center}

\begin{lemma}  \label{lem:wedge}
\begin{description}
\item[(a)]   $ E(Z_{k,h}) \leq c h k^{-\alpha'} $.
\item[(b)]  $ E(Z_{k,h}|\tilde{Z}_{k,h}>0) \geq ch $.
\end{description}
\end{lemma}

\begin{proof} (a) 
By Lemma \ref{lem:wedge} we have
$$ E(Z_{k,h}) = \sum_{t} \sum_{x \in Q_{k,h} } q_t(0,x)^2
= \sum_{t < k^{2+ \alpha'}} h \frac{c}{k^{2(1+\alpha')}}   + \sum_{t \geq k^{2+ \alpha'}}  \frac{ch}{t^{2\beta'} }
\leq \frac{ch}{k^{\alpha'}}.  $$

(b) 
Since we are conditioning on the event $\{\tilde{Z}_{k,h}>0\}$,
there is a collision at position $x=(k,y)$ for some $y$ with $\frac{h}{3} \leq y \leq \frac{2h}{3}$.
Conditioned on this event, the total number of collisions that happen in the set $Q_{k,h}$
will be greater than the number of collisions that take place before the first time that one of the
random walks exits this interval. So, setting $B := Q_{k,h}$, and using \eqref{e:zr} we have
\[
E(Z_{k,h}| \tilde{Z}_{k,h}>0) \geq \frac12 g_{Q_{k,h}}(x,x) =\frac12  R_{\text{eff}}(x,Q_{k,h}^c)\geq c h.
\]
\end{proof}

\begin{proof} [\textsl{Proof of Theorem \ref{comb}(b).}]
By Lemma \ref{lem:wedge}
\[
chk^{-\alpha'} \geq E(Z_{k,h}) \geq P(\tilde{Z}_{k,h}>0) E(Z_{k,h}|\tilde{Z}_{k,h}>0) \geq h P(\tilde{Z}_{k,h}>0),
\]
so that $P(\tilde{Z}_{k,h}>0) \leq k^{-\alpha' }$.
Now summing over all $k$ and over all $h$ ranging over powers of 2 and satisfying $h\leq k^{\alpha}$, we get that
\[
\sum_{k}\sum_{h \text{ powers of 2}} P(\tilde{Z}_{k,h}>0)
\le \sum_{k} \log_2(k^{\alpha}) k^{-\alpha'} < \infty, \text{ since $\alpha'>1$.}
\]
Hence by Corollary \ref{cor:fsets} the total number of collisions is finite almost surely.
\end{proof}

Before we prove  Lemma \ref{lem:qtbound} we give
some \textbf{\textsl{heuristics}} for the bound $E(Z_{k,h}) \leq c h k^{-(\alpha \wedge 2)}$:

The expected time that the random walk takes to reach $k$ on the horizontal axis started
from 0 is of the order $k^{2+\alpha}$. The reason for that is that the expected number of visits by the first coordinate to $i \in
\Z_+$ before hitting $k$ for the first time is $k-i$. At every such visit the walk makes a vertical excursion, which takes time
of order $i^{\alpha}$ in expectation. Hence the total time has expectation which is of order $k^{2+\alpha}$.
This is the right order of the expected time for all $\alpha >1$. The actual time though differs in the
two regimes $1<\alpha <2$ and $\alpha >2$.

The first coordinate makes $k^2$ steps to go from $\frac{k}{2}$ to $k$. When  $1<\alpha <2$, at each step of the horizontal
coordinate we perform an independent experiment. We succeed in each experiment, if we spend time greater than
$k^{2\alpha}$ on the tooth in this step of the first coordinate. The probability of success is then lower bounded by
$\frac{c_1} {k^{\alpha}}$ and in the $k^2$ experiments with high probability there will be a success and the expected number
of successes is $k^{2-\alpha}$, thus the total time taken to reach $k$ will be of order $k^{2+\alpha}$.

When $\alpha >2$ the experiments described above will give us no success with high probability, and so this
method no longer gives us the right order for the hitting time.
In this regime instead we declare a success if we spend time greater than $k^4$ on the tooth. The
expected number of successes is then 1 and thus the total time to reach $k$ is of order $k^4$.

Thus the relevant times that will contribute to the expectation of $Z_{k,h}$ will be of order $k^{2+\alpha'}$.
The probability that the two random walks will have the same horizontal coordinate will be $\left(\frac{1}{k}\right)^2$ and the
probability that they will be at the right height will be $\left(\frac{h}{k^{\alpha'}}\right)^2$ and at the same height will be
$\frac{1}{h}$. We get the uniform distribution, because by that time the random walks will have mixed.

Putting all things together in the formula for the expectation we obtain the aforementioned expression.

The remainder of this section is devoted to the proof of Lemma  \ref{lem:qtbound}.
Our main tool to bound $q_t(0,x)$ will be by comparison with Greens functions.

\begin{lemma} \label{lem:GComp}
Let $B \subset G$. Then
\begin{equation} \label{e:qb1}
 q_t(x,x) \le \frac{ 2g_B(x,x)} { t P_x(\tau_B \geq t) }.
\end{equation}
\end{lemma}

\begin{proof}
The spectral  decomposition \eqref{spectral} shows that $q_{2j}(x,x)$ is decreasing as a function of $j$,
and also that $q_{2j+1}(x,x) \le q_{2j}(x,x)$ for $j \ge 0$.
Using this it is easy to verify that
\begin{align} \label{td}
q_t(x,x) \leq \frac{2}{t}\sum_{j=0}^{t} q_j(x,x).
\end{align}
We now define $g_t(x,x)$ to be the Green kernel until time $t$, i.e. $g_t(x,x) = \sum_{j=0}^{t} q_j(x,x)$.
By the strong Markov property
\[
g_t(x,x) \leq g_{B}(x,x) + P(\tau_{B} < t) g_t(x,x),
\]
where $\tau_{B}$ is the first exit time from the set $B$; rearranging gives \eqref{e:qb1}.
\end{proof}

To use this lemma we wish to choose the set $B$ so that
the Green kernel up to time $t$ and
 the Green kernel of the Markov chain killed after exiting the set $B$ are comparable.
To obtain the necessary bounds on the exit times from the region $B$
we now make precise some of the heuristics given above.





\begin{lemma} \label{lem:hit-tail}
(a) Let $k \ge 0$, $k_1 \ge 1$ and $T= \tau_{H(k-k_1, k+ k_1)}$ be the first exit of $X$ from
$H(k-k_1, k+ k_1)$, where $H(a,b):= \{ (x,y) \in G: a \le x \le b \}$. Then
\begin{align} \label{e:THbnd}
 P_k( T \le t ) \le c \exp( -c (k_1^{2 + \alpha'}/t)^{1/3} ).
\end{align}
(b) Let $k \ge 1$ and $T= \tau_{H(0, k)}$. Then
 \begin{align*}
 P_0( T \le t ) \le c \exp( -c (k^{2 + \alpha'}/t)^{1/3}).
\end{align*}
\end{lemma}

\begin{proof}
Note that (b) follows from (a) by just looking at the random walk from the first
hit on $2k/3$.

Suppose we have \eqref{e:THbnd} when $k_1 \le k$.
Then if $k_1 > k$ we have $H(k-k_1, k+k_1) = H(0, k+k_1)$,
and $\frac{k+k_1}{2} \ge k$. Then since $X$ has to hit $\frac{k+k_1}{2}$ before it leaves $H(0, k+k_1)$, we have
$$  P_k( T \le t ) \le P_{\frac{k+k_1}{2}}( T \le t ) \le c \exp( -c (k_1^{2 + \alpha'}/t)^{1/3} ). $$
Thus it is sufficient to consider the case when $k_1 \le k$.

We now prove (a) in the case when $k_1 \leq k$.  Let $L$
be the number of horizontal steps that the random walk makes until it leaves $H= H(k-k_1/2, k+k_1/2)$.
Choose constants $\lambda >0$ and $\theta \le \frac14$. We have
\begin{align} \label{axis}
P_k( T \le t) \leq P_k( L < k_1^2/  \lambda ) + P_k(T  \le t,  L \ge k_1^2/  \lambda ) .
\end{align}
The first probability appearing on the right hand side of \eqref{axis} is bounded above by
the probability that a simple random on $\Z_+$ travels distance
$k_1/2$ in less than $k_1^2/ \lambda$ steps, which is smaller than $c' \exp(- c'' \lambda)$.

To bound the second probability we are going to perform $N=k_1^2/  \lambda$ independent experiments.
In each experiment we succeed if we hit level $\theta k_1^{\alpha'}$ on the tooth, and then
spend time at least $\theta^2 k_1^{2 \alpha'}$ in the tooth before the next horizontal step.
(The conditions $\theta \le \frac14$ and $k_1/2 \le k/2$ ensure that there is enough room in each tooth.)
Since a simple random walk on $\bZ \cap [0,n]$ started at $m \le n$ has probability at least $c_1$ of taking
more than $m^2$ steps to hit zero, the probability of success for each experiment is at least
$p= c_1/ (\theta k_1^{\alpha'})$.
Thus on the event $\{  L \ge N \}$ we have that $T$ stochastically dominates
$\theta^2 k^{2 \alpha'} \bin(  k_1^2/\lambda , p)$.

Hence
\begin{align} \label{experiment}
P_k( T  \le t,  L \ge k_1^2/  \lambda )
\le P\left(\text{Bin}( k_1^2/  \lambda, p)  \le \frac{t}{ \theta^2 k_1^{2 \alpha'} } \right)
 =  P\left(\text{Bin}( N, p)  \le t'  \right),
\end{align}
where $t' = t / (  \theta^2 k_1^{2 \alpha'})$.

By a straightforward application of Chernoff's bound we have:
\begin{lemma}\label{lem:Ch}
Let $\mu<1$. Then there exists a positive constant $\mu'$ such that
\begin{align*}
P(\bin(n,p) \leq \mu np) \leq e^{-\mu' np}.
\end{align*}
\end{lemma}

Now write $t = \gamma k_1^{2 + \alpha'}$. If $\gamma^{2/3} > (8 c_1)^{-1}$ then by adjusting the constants $c$ the
bound \eqref{e:THbnd} holds. We can therefore assume that  $\gamma^{2/3} \le (8 c_1)^{-1}$.
Let $ \lambda = \gamma^{-1/3}$, and $ \theta = (2/c_1)  \gamma \lambda= (2/c_1) \gamma^{2/3}$;
note that we have $\theta \le \frac14$.
Then if
$$ \mu =  \frac{t'}{Np} =  \frac{\lambda t }{ c_1 \theta k_1^{2+\alpha'}} =
 \frac{\eps \lambda }{c_1 \theta} = \frac12, $$
Lemma \ref{lem:Ch} gives
\begin{equation} \label{e:Tb2}
 P_k( T  <t,  L \ge k_1^2/  \lambda )  \le e^{-c Np} \le
 \exp( -c k_1^{2-\alpha'} (c_1^2/ 2 ) \gamma^{-1/3} )
  \le e^{-c' \gamma^{-1/3} }.
 \end{equation}
Thus both terms in \eqref{axis} are bounded by terms of the form $c \exp(-c' \gamma^{-1/3})$.
\end{proof}

%



\begin{lemma}\label{claim1}
$q_t(x,x) \le \frac{c}{t^{\beta'}}$ for any $x=(k,0)$ on the horizontal axis and $t\ge 1$.
\end{lemma}

\proof
Let $k_1 = b t^{1/(\alpha' +2 )}$, where $b \ge 1$ is a constant which will be chosen later.
We use Lemma \ref{lem:GComp} with
$$ B = H(k-k_1, k+k_1) = \{ (x,y) \in G: k - k_1  \leq x \leq k + k_1 \}. $$
Then $\Reff(x, B^c) \le c k_1$. By Lemma \ref{lem:hit-tail}
\begin{equation} \label{e:t11}
P_k( \tau_B < t) \le  c_1 \exp( -c_2 (k_1^{2 + \alpha'}/t)^{1/3} )
= c_1  \exp(-c_2 b^{(2+\alpha')/3} ).
\end{equation}
Taking $b$ large enough, the right side of \eqref{e:t11} can be made less than $1/2$.
Hence by Lemma \ref{lem:GComp}
$$ q_t(x,x) \le c t^{-1} \Reff(x, B^c) \le c t^{-1} k_1 \le c' t^{-\beta'}. $$
\qed

\begin{lemma} \label{lem:offh}
$q_t(0,(k,h)) \le c t^{-\beta' } e^{-\frac{h^2}{c't}}$, for all points $(k,h)$ and all times $t>0$.
\end{lemma}

\proof
Let $T_0$ be the first hitting time of 0 for a simple random walk on $\Z$.
Using the ballot theorem (see, e.g. \cite{Ballot}) we get
\[
P_h(T_0 = s) = \frac{h}{s} P_h(S_s = 0) \leq c \frac{h}{s} \frac{1}{\sqrt{s}} e^{-\frac{h^2}{c's}}.
\]
Let $T_0^m$ be the first hitting time of 0 for a simple random walk restricted to the interval $[0,m]$.
Then $T_0$ stochastically dominates
$T_0^m$. 
By reversibility we have $q_t(0,(k,h)) = q_t((k,h),0)$ and so by Lemma \ref{claim1}
\begin{align*}
q_t( (k,h), 0) \le
\sum_{s=1}^{t-1} P_h(T_0^{k^{\alpha}} =s) q_{t-s}((k,0),0)
\leq \sum_{s=1}^{t-1} P_h(T_0^{k^{\alpha}} =s) c(t-s)^{-\beta'}.
\end{align*}
Set $\psi(s) = c(t-s)^{-\beta'}$.
Then $\psi$ is an increasing function and thus using the stochastic monotonicity we mentioned above
we have that $\sum_{s}P_h(T_0^{k^{\alpha}} =s)\psi(s) \leq \sum_{s}P_h(T_0=s)\psi(s)$. So,
\begin{align*}
q_t( (k,h), 0)&\leq \sum_{s=1}^{t-1} P_h(T_0=s)  \frac{c}{(t-s)^{{\beta'}}}
 \leq \sum_{s=1}^{t-1} c \frac{h}{s} \frac{1}{\sqrt{s}} e^{-\frac{h^2}{c's}} \frac{1}{(t-s)^{{\beta'}}} \\
&\leq h c \left( \sum_{s=1}^{\frac{t}{2}} \frac{1}{s^{\frac{3}{2}}} \frac{1}{t^{{\beta'}}}e^{-\frac{h^2}{c's}} + \sum_{\frac{t}{2}\leq s\leq t-1} \frac{1}{t^{\frac{3}{2}}} \frac{1}{(t-s)^{{\beta'}}}e^{-\frac{h^2}{c's}}  \right)  \\
&\leq hc''\left(\frac{1}{\sqrt{t}}e^{-\frac{h^2}{c't}} \frac{1}{t^{{\beta'}}} + \frac{t^{\frac{1}{4}}}{t^{\frac{3}{2}}}e^{-\frac{h^2}{c't}} \right).
\end{align*}
But $h e^{-\frac{h^2}{c't}} \leq c_1 \sqrt{t}e^{-\frac{h^2}{2c't}}$, so
\begin{align*}
q_t( (k,h), 0)  \leq C t^{-\beta'} e^{-\frac{h^2}{ct}}.
\end{align*}
\qed

\begin{lemma} \label{smalltimes2}
Let $x = (k,0)$. Then if $t< k^{2+\alpha}$, we have
\[
q_t(0,x) \leq \frac{c}{k^{1+\alpha}} = c\left(k^{2+\alpha}\right)^{-\beta}.
\]
Hence
$$ \sup_{ t\ge 0} q_t(0,x) \le c\left(k^{2+\alpha}\right)^{-\beta}. $$
\end{lemma}

\proof
Let $m$ be an integer within distance $1$ of $k/2$, and let $T= T_m$.
Now
\begin{equation} \label{e:split}
P_0(X_t=x) = P_0(X_t=x, T_m \le t/2) + P_0 (X_t=x, T_m \ge t/2).
\end{equation}
By time reversibility, we have that
\[
P_0(X_t=x, T_m > t/2 ) = c P_x (X_t = 0, \text{ last visit to $m$ before } t/2 )
\le c P_x ( X_t=0, T_m < t/2),
\]
so to bound the second term in \eqref{e:split} it suffices to bound $P_x(X_t=0, T_m \leq t/2)$.

By the strong Markov property we have
$$ P_0(X_t=x, T_m  \leq t/2 )
   \le P_0 (T_m  \le t/2) \max_{0 \le s \le t/2} P_m (X_{t-s} = x). $$
We bound the first term above using Lemma \ref{lem:hit-tail}, while the second term
is bounded by $c t^{-\beta}$. Thus writing $t =  k^{2+\alpha}/\eta$, we have
\begin{align*}
 P_0(X_t=x, T_m  \leq t/2 ) &\le c t^{-\beta} \exp( -c ( k^{2+\alpha} /t)^{1/3} ) \\
 &\le c k^{-1-\alpha} \eta^{\beta} e^{ -c \eta^{1/3} } \\
 &\le  c k^{-1-\alpha}  \sup_{\eta>0} (\eta^{\beta} e^{ -c \eta^{1/3} } ) \le c'  k^{-1-\alpha} .
\end{align*}
The term $P_x(X_t=0, T_m \leq t/2)$ is bounded in exactly the same way.
\qed

\begin{proof}[\textsl{Proof of Lemma \ref{lem:qtbound}.}]
By Cauchy-Schwartz we have:
\[
q_t(0,x) \leq \sqrt{q_t(0,0)} \sqrt{q_t(x,x)}.
\]
Suppose first $x=(k,0)$ is on the horizontal axis.
Then the bound on $q_t(0,x)$ follows from Lemma \ref{claim1}
if $t \ge k^{2 +\alpha'}$, and from Lemma \ref{smalltimes2} if  $t \le k^{2 +\alpha'}$.

If $x = (k,h)$ where $h>0$ then the bound follows from Lemma \ref{lem:offh}
if $t \ge k^{2+\alpha'}$. If  $t \le k^{2+\alpha'}$ then by considering the first hit on $k$
$$ P_x( X_t =0) \le \max_{0\le s \le t} P_k(X_s =0) \le c k^{-1-\alpha'} $$
by Lemma \ref{smalltimes2}.
\end{proof}


\section{Spherically symmetric trees} \label{sec:sstree}

In this section we are going to show the double phase transition taking place in the spherically symmetric trees of
lengths $b_n=2^{2^{\beta n}}$. We remark that for part (c) the Green kernel criterion, Theorem \ref{greenkernel} does not apply.

\begin{proof}[\textsl{Proof of Theorem \ref{trees}(a)}]
Let $(X_n,Y_n)_n$ be a discrete time walk on the product space $T \times T$.
To show recurrence of the pair $(X_n,Y_n)$, we are going to use the Nash-Williams criterion of recurrence,
which can be found for instance in \cite{Mixing} (Chapter 21, Proposition 21.6).

Let $\Pi_n=\{ x \in T: d(0,x) =n\}$, and
\[
\Pi_n^* = (\Pi_n \times \cup_{i \leq n}\Pi_i) \cup (\cup_{i \leq n}\Pi_i \times \Pi_n) \subset T\times T.
\]
Let $E_n$ be the set of edges with at least one vertex in $\Pi^*_n$.
Then the sets $(E_{2n})_n$ constitute a sequence of disjoint edge-cutsets that
separate $(o,o)$ from $\infty$. To show recurrence of $(X_n,Y_n)_n$,
by the Nash-Williams criterion of recurrence (see for instance \cite{Mixing} (Chapter 21, Proposition 21.6))
we only need to show that
\be \label{e:pisum}
\sum_n |E_{2n}|^{-1}  = \infty.
\ee
We have $|E_n| \le c |\Pi_n^*| \le c' |\Pi_n| \times (\sum_{i=1}^{n} |\Pi_i|)$.
However $|\Pi_n| \asymp (\log n)^{\frac{1}{\beta}}$, and so
$\sum_{i=1}^{n} |\Pi_i| \asymp \sum_{i=1}^{n}(\log i)^{\frac{1}{\beta}} \le n (\log n)^{\frac{1}{\beta}}$. Hence
\[
|\Pi_n^*| \le  C n (\log n)^{\frac{2}{\beta}},
\]
and therefore as $\beta \ge 2$  \eqref{e:pisum} diverges.
\end{proof}

The remaining parts of the proof will require estimates of the transition probabilities of the
random walk $X$ on $T$. Since
it will sometimes be convenient to use these rather than the transition density $q_t(x,y)$
we write
$$ p_t(x,y) = P_x(X_t = y). $$

Let
\be
 a_n = \sum_{i=0}^{n-1} b_i, \quad n \ge 1.
\ee
Note that the $n$-th branch point from $o$ is at distance $a_n$ from $o$.
For $x \in T$ let $n(x)$ be the number of branches at the same level as $x$;
$n(x)$ is also the number of vertices $y \in T$ such that $d(o,y)=d(o,x)$.
We write
$$ J_n = \{ x \in T: n(x) = 2^n \}; $$
these are the points between the $(n-1)$ -th and $n$-th branch points.

\begin{remark}\label{rembd}
{\rm
Our main tool will be by comparison with a birth and death chain $\tilde X$  on $\Z_+$ that jumps to
either $x+1$ or $x-1$ with the following probabilities. If $x$ is at distance $a_n$
from the origin for some $n\ge 1$, then
$p_{x,{x+1}} = \frac{2}{3} = 1 - p_{x,x-1}$, otherwise for all other $x$, $p_{x,{x+1}} = \frac{1}{2} = 1 - p_{x,x-1}$.
We write  $p_t^{\text{BD}}(o,x')$ for the transition probabilities of this birth and death chain,
and $q_t^{\text{BD}}(o,x')$ for its transition density with respect to its invariant measure.
Note that $X'_t = d(o, X_t)$ has the law of this birth and death chain. Therefore if
for $x \in T$ we write $|x|=d(o,x)$ then by symmetry
\begin{align}
\label{birthdeath}
p_t(o,x) = \frac{1}{2^{n(x)}} p_t^{\text{BD}}(o,|x|).
\end{align}
We write $\tau'$, $T'$ etc. for hitting and exit times for the birth and death chain.
} \end{remark}

\begin{lemma} \label{claim12}
We have for $t \ge 0$
\be  \label{e:pt00}
p_t(o,o) \leq \frac{c}{\sqrt{t} (\log{t})^{1/ \beta}},
\ee
and for $x \in J_n$, $t \ge 0$,
\be \label{e:ptxx}
 p_t(o,x) \le \frac{     c_1}
 { \sqrt{t}  {\left(\log{t}\right)^{{1}/{2\beta}} \left(\log(|x|+c_2 \sqrt{t})\right)^{{1}/{2\beta}} }}.
\ee
\end{lemma}

\proof
Let $B = \{ y \in \Z_+: y \leq |x|+b\sqrt{t}\}$, for a constant $b$ to be determined later.
Applying Lemma \ref{lem:GComp} to the birth and death chain $X'$
\begin{align*}
 q_t^{\text{BD}}(|x|,|x|) \leq \frac{c''R_{\text{eff}}(|x|,B^c)}{tP_{|x|}(\tau'_{B} \geq t)}.
\end{align*}
It is easy to verify that
$R_{\text{eff}}(|x|,B^c) \le 2^{-n-m} 2b\sqrt{t}$, where $m$ is the number of branch points
between $|x|$ and $|x|+c\sqrt{t}$; note that at each branch point the effective resistance is halved.
Since there are approximately $\beta^{-1} \log_2 \log_2 r$ branch points between $0$ and $r$,
we have
$m = \frac{1}{\beta}\left(\log_2 \log_2(|x|+b\sqrt{t}) - \log_2 \log_2 |x| \right)$,
if $x \neq o$ and $\frac{1}{\beta}\log_2 \log_2(b\sqrt{t})$ if $x=o$.
We now need to bound $P_{|x|}(\tau'_B < t)$.
Since for each $n$, $\sum_{k=1}^{n}b_k \asymp 2^{2^{\beta n}}$, there must exist a branch of length at
least $\frac{1}{2}b \sqrt{t}$ between $|x|$ and $|x|+b\sqrt{t}$; call this branch $A$. Let
$y$ be the midpoint of $A$. Then
$P_{|x|}(\tau'_{B} < t)$ is smaller than the probability that a simple random walk started
at $y$ remains in $A$ for time at least $t$.
But from the exponential hitting time bounds for the simple random walk on $\Z$
we can make this probability as small as we like by choosing the constant $b$ large.
So, taking $b=c_2$ large enough we have $P_{|x|}(\tau'_{B} \geq t) >1/2$ and hence
since $2^n \simeq (\log_2 |x|)^{1/\beta}$,
\be \label{e:qbd-1}
 q_t^{\text{BD}}(|x|,|x|) \le c'' 2^{-n} t^{-1/2}  \left(\frac {\log_2(|x|+c_2\sqrt{t})} {\log_2 |x|} \right)^{ -1/\beta }
 \le c''  t^{-1/2}  \left( \log_2( |x|+c_2\sqrt{t})  \right)^{ -1/\beta } .
\ee
A similar calculation gives
$$ q_t^{\text{BD}}(0,0) \le c'' t^{-1/2}  \left(  {\log( c_2\sqrt{t})} \right)^{ -1/\beta } . $$
Hence
$$ q_t^{\text{BD}}(0,|x|) \le c  t^{-1/2} \left( \log_2( |x|+c_2\sqrt{t})  \right)^{ -1/2 \beta }
 \left(  {\log( c_2\sqrt{t})} \right)^{ -1/2 \beta } . $$
 Since $p_t^{\text{BD}}(0,x) = 2^n q_t^{\text{BD}}(0,x)$,
 using \eqref{birthdeath} we have $p_t(o,x) = q_t^{\text{BD}}(0,|x|)$ and this completes
 the proof.
 \qed

\begin{proof}[\textsl{Proof of Theorem \ref{trees}(b)}]
Transience of the product chain is equivalent to the sum
$\sum_{t} p_t(o,o)^2$ being finite. Using the upper bound \eqref{e:pt00}
we get that
\begin{align}
\sum_{t} p_t(o,o)^2 \leq \sum_{t} \frac{c}{t (\log{t})^{\frac{2}{\beta}}},
\end{align}
which is finite since $\beta <2$.
\end{proof}

\begin{proof}[\textsl{Proof of Theorem \ref{trees}(c)}]
Let $X$ and $Y$ be two independent discrete time simple random walks on the tree $T$.
We are going to count the number of collisions that occur at level $n$, i.e. on all the segments of length $b_n = 2^{2^{\beta n}}$.
Set $a = \sum_{i=0}^{n-1} b_i$, which is the distance from the root to the $(n-1)$-th branch point. We
now divide the segment of length $b_n$ into subintervals.
The first one has length equal to $2^0a$, the second one $2^1a$ and the $l$-th
one has length $2^{l-1}a$. In total we get order $2^{\beta (n-1)}$ such intervals, say $\alpha 2^{\beta(n-1)}$. Let $I_{n,l}^i$ denote the $l$-th such
interval on the $i$-th branch, for $i=1,\cdots, 2^n$ and let $J_{n,l}$
denote the collection of all these subintervals, i.e. $2^n$ in total.
We are going to divide the proof into two parts: for $\beta \geq 1$ and $\frac{1}{2} \leq \beta <1$, primarily because the relevant times that contribute to the number of collisions are of different orders, but also for some other technical reasons.

\framebox{$\beta \geq 1$}

We define
\begin{align} \label{collision}
Z_{n,l} = \sum_{t=(2^la)^2}^{2(2^la)^2} \1(X_t = Y_t \in J_{n,l}).
\end{align}
Thus $Z_{n,l}$ counts the number of collisions that happen on the set $J_{n,l}$ and at times
that are of order $(2^la)^2$.
We want to lower bound $P_o(Z_{n,l}>0)$.
To do so, we are going to lower bound $E_o(Z_{n,l})$, upper bound $E_o(Z_{n,l}|Z_{n,l}>0)$ and then use the obvious
equality
\begin{align}
\label{obvious}
P_o(Z_{n,l}>0) = \frac{E_o(Z_{n,l})}{E_o(Z_{n,l}|Z_{n,l}>0)}.
\end{align}

\begin{lemma}
\label{lemma}
There exists a constant $c>0$ such that $P_o(X_t \in I_{n,l}^i) \geq \frac{c}{2^n}, \forall i$,
for $t$ such that $(2^la)^2 \leq t\leq2(2^la)^2$
and for all $l\geq 1$.
\end{lemma}

\proof
Let $\tilde{X}$ be the birth and death chain described in Remark \ref{rembd}. Then we can couple it with a simple random walk $S$ on $\Z_+$, such that
$\tilde{X}_t \geq S_t$, for all $t$. Let $b$ be the last branch point (which is at distance $a$ from 0) before the interval $I_{n,l}^i$. Then we can couple $\tilde{X}$ with a simple random walk $\tilde{S}$ started from $|b|$ with state space $[|b|, \infty)$ and such that $\tilde{X}_t \leq \tilde{S}_t$, for all $t$. We then have
\begin{align*}
P_0(\tilde{X}_t \in (2^la,2(2^la))) = P_0(\tilde{X}_t \leq 2(2^la)) - P_0(\tilde{X}_t \leq 2^la) \geq P_{|b|}(\tilde{S}_t \leq 2(2^la)) -
P_0(S_t \leq 2^la).
\end{align*}
But $P_{|b|}(\tilde{S}_t \leq 2(2^la)) = P_0(S'_t + a \leq 2(2^la)) =  P_0(S'_t \leq 2(2^la) - a)$, where $S'$ is a simple random walk on $\Z_{+}$ started
from 0. For $l\geq 1$ we then have that $x:=2(2^la) - a > a$ and so we get
\begin{align*}
P_b(\tilde{S}_t \leq 2(2^la)) - P_0(S_t \leq 2^la) = P_0(S'_t \in (a,x)).
\end{align*}
But since $t$ is of order $(2^la)^2$, we get that $\exists c>0$ such that
such that $P_0(S'_t \in (2^la,x)) \geq c$, so for the birth and death chain we obtain that
\[
P_0(\tilde{X}_t \in (2^la,2(2^la))) \geq c,
\]
for all $l\geq 1$ and hence using \eqref{birthdeath} we deduce
\begin{align*}
P_o(X_t \in I_{n,l}^i) \geq \frac{c}{2^n}.
\end{align*}
\qed

\begin{claim}
$E_o(Z_{n,l}) \geq C \frac{2^la}{2^n}$.
\end{claim}
\proof
By symmetry we have
\begin{align*}
E_o(Z_{n,l}) &= 2^n \sum_{x \in I_{n,l}}\sum_{t=(2^la)^2}^{2(2^la)^2} p_t(o,x)^2
\geq 2^n \sum_{t=(2^la)^2}^{2(2^la)^2} \frac{1}{|I_{n,l}|}
\left(\sum_{x \in I_{n,l}} p_t(o,x)\right)^2 \\&
= 2^n \sum_{t=(2^la)^2}^{2(2^la)^2} \frac{1}{|I_{n,l}|} P_o(X_t \in I_{n,l})^2 \geq C \frac{2^l a}{2^n},
\end{align*}
where for the first inequality we used Cauchy-Schwartz and for the last one we used Lemma \ref{lemma}.
\qed

\begin{claim}
$E_o(Z_{n,l}|Z_{n,l}>0) \leq C' 2^la$.
\end{claim}

\proof
Since we are conditioning on the event $\{Z_{n,l}>0\}$, there is a collision on one of the subintervals $I_{n,l}$.
Starting from this point, we are counting all the collisions that happen for times $ (2^la)^2 \leq t \leq 2(2^la)^2$.

We first count the number of collisions that occur before the first time that one of the random walks exits the set
$A_{n,l} = I_{n,l-1} \cup I_{n,l} \cup I_{n,l+1}$.  By Lemma \ref{lem:GZB} this number is bounded by the
effective resistance from the starting point to $A_{n,l}^c$,
which is bounded by $2^{l+1}a$, no matter where in the interval $I_{n,l}$ the random
walks started. We then wait until the next time that both of the random walks
have a collision in one of the intervals $I_{n,l}$. Starting from there
we again wait for one of them to exit the set $A_{n,l}$,
and then we upper bound the number of collisions by $2^{l-1}a$. The total number of
rounds that we can have has expectation bounded by a constant. This is because, once a random walk
is in the interval $I_{n,l}$ it has to travel distance at least $2^{l-1}a$ in order to exit $A_{n,l}$.
Thus the time it takes has expectation at least $(2^{l-1}a)^2$. Since
we are interested only in collisions that happen in a time interval of length $(2^la)^2$ we
deduce that the total number of rounds has bounded expectation.

Hence we conclude that
\[
E_o(Z_{n,l}|Z_{n,l}>0) \leq C' 2^la.
\]
\qed

Using \eqref{obvious} we obtain
\begin{equation}
\label{basic2}
P_o(Z_{n,l}>0) \geq \frac{c}{2^n}.
\end{equation}

Let $Z_n = \sum_{l=1}^{\alpha 2^{\beta(n-1)} - 1} \1(Z_{n,l}>0)$, i.e. $Z_n$ counts the number of subintervals of $b_n$ except the first and
last one, where there is at least one collision.
Using \eqref{basic2}, we get that $E_o(Z_n) \geq c 2^{(\beta-1)n}$. We want to lower bound $P_o(Z_n >0)$ and we will use the second moment estimate
\begin{equation}
\label{moment}
P_o(Z_n > 0) \geq \frac{(E_o(Z_n))^2}{E_o(Z_n^2)}.
\end{equation}

\begin{claim}
$E_o(Z_n^2) \leq c' 2^{2(\beta-1)n}$.
\end{claim}

\proof
We have that
\begin{align}
\label{second}
E_o(Z_n^2) \leq 2\sum_{l=1}^{\alpha 2^{\beta(n-1)} - 1} P_o(Z_{n,l}>0) + \sum_{l=1}^{\alpha 2^{\beta(n-1)} - 1}
\sum_{k=2}^{\alpha 2^{\beta(n-1)} - 1 - l} P_o(Z_{n,l}>0,Z_{n,l+k}>0).
\end{align}

Write $A_{n,l}^i$ for the event that $I_{n,l}^i$ is visited
by one simple random walk in the time interval we are interested in.
Let
$$ N = \sum_{i=1}^{2^n} \1(A_{n,l}^i). $$
Then $E(N) \leq c$, for a positive finite constant $c$, since once such an
interval is visited then the walk has to travel distance of order
$2^la$ in order to reach a branch point and then visit another interval and that time has
expectation greater than $c'(2^la)^2$.
Thus, using the symmetry of the tree, we have that for any $i$,
\begin{equation} \label{e:ani}
P_o(A_{n,l}^i)\leq \frac{c}{2^n}.
\end{equation}
Hence,
\[
P_o(Z_{n,l}>0) \leq \sum_{i=1}^{2^n} P_o(A_{n,l}^i)^2 \leq \frac{c}{2^n},
\]
and thus the first term on the right hand side of \eqref{second} is upper bounded by $2^{(\beta-1)n}$.

For the second term we have $P_o(Z_{n,l}>0,Z_{n,l+k}>0) = P_o(Z_{n,l+k}>0|Z_{n,l}>0) P_o(Z_{n,l}>0)$ and
\begin{align*}
P_o(Z_{n,l+k} &>0 |Z_{n,l}>0)\\
& =P_o(Z_{n,l+k}>0, \text{ at least 1 of the RWs hits $J_{n,l+k}$ before o }|Z_{n,l}>0)\\
&\qquad + P_o(Z_{n,l+k}>0, \text{ both hit $o$ before $J_{n,l+k}$ }|Z_{n,l}>0).
\end{align*}

To upper bound the first term, let $p_r$ denote the probability that the random walk $X$ starting from the
set $J_{n,l}$ goes back through exactly  $r$ branch points towards the origin before it
first hits $J_{n, l+k}$.
Then the probability starting from $J_{n,l}$ that $X$ reaches $J_{n,l+k}$ before hitting $o$ is bounded
from above by:
\begin{align}
\nonumber
\sum_{r=1}^{n} p_r &P(\text{starting from the $(n-r)$-th b.p. $X$  hits $J_{n,l+k}$ } \\
\label{gambler}
& \qquad \qquad \qquad \text{before it hits the $(n-r-1)$-th b.p.}).
\end{align}
The probability appearing in the sum above can be computed as follows: starting from the $(n-r)$-th b.p.
there are $2^r$ intervals $I_{n,l+k}$ that we can hit before hitting the $(n-r-1)$-th b.p..
We fix one such interval. Then the probability that we hit that before the $(n-r-1)$-th b.p. is given by
the gambler's ruin probability and is equal up to constants to
$\frac{1}{2^k \left(b_{n-1}\right)^{1-2^{-\beta r}}}$, so the above sum becomes
\begin{align}
\sum_{r=1}^{n} p_r \frac{2^r}{2^k \left(b_{n-1}\right)^{1-2^{-\beta r}}}
\le \sum_{r=1}^{n}  \frac{2^r}{2^k \left(b_{n-1}\right)^{1-2^{-\beta r}}}
\leq \frac{c}{2^k}.
\end{align}

For the second term we have, using \eqref{e:ani},
\begin{align*}
P_o(Z_{n,l+k}>0, \text{ both hit $o$ before $J_{n,l+k}$ }|Z_{n,l}>0) \leq
\sum_{i=1}^{2^n} P_o(A_{n,l+k}^i)^2 \leq \frac{c}{2^n}.
\end{align*}

So putting these estimates together we get
\begin{align*}
P_o(Z_{n,l}>0,Z_{n,l+k}>0) \leq \frac{c}{2^n} \left( \frac{c'}{2^k} + \frac{c''}{2^n} \right).
\end{align*}

Hence $E_o(Z_n^2) \leq c' 2^{2(\beta-1)n}$, since $\beta>1$.
\qed

Using \eqref{moment} we obtain that
\[
P_o(Z_n > 0) \geq c>0.
\]
Hence by Corollary \ref{cor:fsets} we have $P(Z=\infty) =1$; this completes the proof of Theorem \ref{trees}(c) in the case $\beta \geq 1$.


\framebox{$\frac{1}{2} \leq \beta <1$}

We now define
\begin{align*}
Z_{n,l} = \sum_{t=2(2^{l+1}a)^2}^{(2^{l+1}a)^4} \1(X_t = Y_t \in J_{n,l}),
\end{align*}
i.e. we are now looking at much longer time intervals.
We want to upper bound the probability that there is a collision in the set $J_{n,l}$, i.e. $P_o(Z_{n,l}>0)$.
To do so we are going to use again the equality
\begin{align}
\label{tilde}
P_o(Z_{n,l}>0) = \frac{E_o(Z_{n,l})}{E_o(Z_{n,l}|Z_{n,l}>0)},
\end{align}
so we need to upper bound $E_o(Z_{n,l})$ and lower bound $E_o(Z_{n,l}|Z_{n,l}>0)$. To do so, we are going to obtain upper and lower bounds
for the transition probabilities in $t$ steps.

\begin{lemma}
\label{lowerbound}
$p_t(o,x) \geq \frac{1}{2^n}\frac{c_1}{\sqrt{t}} \frac{\left(\log{x}\right)^{\frac{1}{\beta}}}{
\left(\log(c_2 \sqrt{t})\right)^{\frac{1}{\beta}}}$, for all $x \in I_{n,l}$ and all  $t > 2 (2^{l+1} a)^2$.
\end{lemma}
\proof
We will again show the lower bound for the birth and death chain and then dividing through by $2^n$ we will get the
transition probability for the tree.
\begin{align*}
p_t^{\text{BD}}(0,x) = P_0(T_x <t, X_t =x)
&\geq P_0(T_x<t) \min_{2s\leq t}p_{2s}(x,x) \\
& \geq P_0(T_x<t) (p_t^{\text{BD}}(x,x) + p_{t-1}^{\text{BD}}(x,x)),
\end{align*}
since $p_{2s}(x,x)$ is a decreasing function of $s$.
Let $Q_t = \{ y \in \Z_+: y \leq x+c\sqrt{t} \}$. Then by Cauchy-Schwartz we have
\begin{align*}
p_{2t}^{\text{BD}}(x,x) &= \sum_{y} p_t^{\text{BD}}(x,y) p_t^{\text{BD}}(y,x)
= \sum_y p_t^{\text{BD}}(x,y)^2 \frac{d(x)}{d(y)} \\
& \geq \sum_{y \in Q_t} p_t^{\text{BD}}(x,y)^2 \frac{d(x)}{d(y)}
\geq \frac{d(x)}{|Q_t|}\left(\sum_{y \in Q_t} \frac{p_t^{\text{BD}}(x,y)}{\sqrt{d(y)}}\right)^2.
\end{align*}
For $y\in Q_t$ we have $d(y)\leq \left(\log(x+c\sqrt{t})\right)^{\frac{1}{\beta}}$. Also $|Q_t| = x + c\sqrt{t} \leq c_1 \sqrt{t}$, so
\begin{align*}
p_{2t}^{\text{BD}}(x,x)
\geq \frac{\left(\log{x}\right)^{\frac{1}{\beta}}}{c_1\sqrt{t}\left(\log(c_1\sqrt{t})\right)^{\frac{1}{\beta}}}
P_x(X_t \in Q_t)^2
\end{align*}
and $P_x(X_t \in Q_t) > c'>0$ by the same argument we used in the proof of Claim \ref{claim12},
i.e. by bounding it a by simple random walk on the last segment with no branch points.
Also $P_0(T_x<t)\geq \frac{1}{2}$, since $t > 2x^2$ and we can bound the
birth and death chain from below by a simple random walk on $\Z_+$.
\qed

\begin{claim}
$E_o(Z_{n,l}) \geq c \frac{|I_{n,l}| \log{|I_{n,l}|}}{2^n}$.
\end{claim}
Using the lower bound for the transition probabilities we get
\begin{align*}
E_o(Z_{n,l}) \geq 2^n\sum_{t=2|I_{n,l}|^2}^{|I_{n,l}|^4} \sum_{x \in I_{n,l}} p_t(o,x)^2
\geq \frac{c}{2^n}\sum_{x \in I_{n,l}} \sum_{t=|I_{n,l}|^2}^{|I_{n,l}|^4} \frac{1}{t}
\frac{(\log{x})^{\frac{2}{\beta}}}{(\log{t})^{\frac{2}{\beta}}}
\geq c\frac{|I_{n,l}|\log{|I_{n,l}|}}{2^n}.
\end{align*}
\qed

\begin{claim}
$E_o(Z_{n,l}|Z_{n,l}>0) \leq c |I_{n,l}|$.
\end{claim}
\proof
Since we are conditioning on the event $\{Z_{n,l}>0\}$, there is a collision on one of the subintervals $I_{n,l}$.
Starting from this point, we are counting all the collisions that happen for times $ 2(2^la)^2 \leq t \leq (2^la)^4$.

We first count the number of collisions that occur before the first time that one of the random walks exits the set
$A_{n,l} = J_{n,l-1} \cup J_{n,l} \cup J_{n,l+1}$, for $l\geq 1$.  This number is up to constants equal to the effective resistance from the starting point
to $A_{n,l}^c$, which is bounded by a constant times $2^{l}a=|I_{n,l}|$, no matter where in the interval $I_{n,l}$ they started from. We define a round as follows:
it starts when there is a collision and it ends when one of the walks exits the set $A_{n,l}$. The number of rounds we have before either of
the two random walks hits zero has bounded expectation. This is because, starting from $I_{n,l}$ the probability that
after exiting $A_{n,l}$ we visit the root before returning to the set $J_{n,l}$ is greater than a constant, for $n \geq n_0$. This follows by bounding the complementary probability by the sum appearing in \eqref{gambler}.
Hence the number of rounds before
hitting the root has a Geometric distribution, so it has bounded expectation. The number of collisions per such round is bounded
from above by $c|I_{n,l}|$ as we argued above.

Hence so far we have considered only those rounds where none of the walks  hits the root before returning to $I_{n,l}$. For the
total number of collisions though we have to consider also those that occur after one of the walks hits the root. But this number will be bounded
by the total number of collisions that occur in $J_{n,l}$ in the time interval of interest. Since one of the walks starts from the root,
if we count the total number of collisions that happen in $J_{n,l}$ for the birth and death chain, then by uniformity we have
to divide through by $2^n$ to get the total number of collisions on the tree.

For the birth and death chain the number of collisions when one walk starts from 0 and the other one from $y$ will be bounded by
\begin{align}
\label{collisionbd}
\sum_{t=2(2^{l+1}a)^2}^{(2^{l+1}a)^4} \sum_{x \in I_{n,l}} p_t^{\text{BD}}(0,|x|) p_t^{\text{BD}}(y,|x|).
\end{align}

We have that $p_t^{\text{BD}}(y,z) = q_t^{\text{BD}}(y,z) 2^{n(z)}$, where we recall $n(z)$ is the number of branch points between 0 and $z$. So by Cauchy Schwartz we get that $p_t^{\text{BD}}(y,z) \leq 2^{n(z)} \sqrt{q_t^{\text{BD}}(y,y)} \sqrt{q_t^{\text{BD}}(z,z)} $ and thus using \eqref{e:qbd-1} we obtain that  $p_t^{\text{BD}}(y,x) \leq \frac{c(\log{x})^{\frac{1}{\beta}}}{\sqrt{t} \log{t}}$, so
the sum \eqref{collisionbd} is bounded from above by $|I_{n,l}| \log{|I_{n,l}|}$, hence transferring back to the tree we get
that
\begin{align*}
E_o(Z_{n,l}|Z_{n,l}>0) \leq c|I_{n,l}| + \frac{|I_{n,l}| \log{|I_{n,l}|}}{2^n} = c|I_{n,l}| + \frac{|I_{n,l}| (l+2^{\beta(n-1)})}{2^n}
\leq c' |I_{n,l}|,
\end{align*}
since $\beta <1$ and $l < 2^{\beta(n-1)}$.
\qed

Hence using \eqref{tilde} we get that
\begin{align}
\label{basic1}
P_o(Z_{n,l}>0) \geq c \frac{2^{\beta(n-1)}}{2^n}
\end{align}

Let $Z_n = \sum_{l=1}^{\alpha 2^{\beta(n-1)} - 1} \1(Z_{n,l}>0)$, i.e. $Z_n$ counts the number of subintervals of $b_n$ except the first and
last one, where there is at least one collision.
Using \eqref{basic1}, we get that $E_o(Z_n) \geq c 2^{(2\beta-1)n}$. We want to lower bound $P_o(Z_n >0)$. To this end we are going to use
the second moment method, i.e.
\begin{equation}
\label{moment2}
P_o(Z_n > 0) \geq \frac{(E_o(Z_n))^2}{E_o(Z_n^2)}.
\end{equation}

\begin{claim}
\label{notation}
$E_o(Z_n^2) \leq c' 2^{2(2\beta-1)n}$.
\end{claim}

\proof
For the second moment we have that
\begin{align}
\label{secondmoment}
E_o(Z_n^2) \leq 2\sum_{l=1}^{\alpha 2^{\beta(n-1)} - 1} P_o(Z_{n,l}>0) + \sum_{l=1}^{\alpha 2^{\beta(n-1)} - 1}
\sum_{k=2}^{\alpha 2^{\beta(n-1)} - 1 - l} P_o(Z_{n,l}>0,Z_{n,l+k}>0).
\end{align}
We let $A_{n,l} = J_{n,l-1} \cup J_{n,l} \cup J_{n,l+1}$, for $l=1,\cdots,\alpha 2^{\beta(n-1)} -1$ and for $l=0$ we define
$A_{n,0} = J_{n-1,\alpha 2^{\beta(n-2)}} \cup J_{n,0} \cup J_{n,1}$ and for $l=\alpha 2^{\beta(n-1)}$ we let
$A_{n,l} = J_{n,l-1} \cup J_{n,l} \cup J_{n+1,0}$. We now define
$\tilde{Z}_{n,l}= \sum_{t=2(2^{l+1}a)^2}^{(2^{l+1}a)^4} \1(X_t = Y_t \in A_{n,l})$ and we have that
\begin{align}
\label{tilde2}
P_o(Z_{n,l}>0) \leq \frac{E_o(\tilde{Z}_{n,l})}{E_o(\tilde{Z}_{n,l}|Z_{n,l}>0)}.
\end{align}
Using the upper bounds for the transition probabilities we get that
\begin{align*}
E_o(\tilde{Z}_{n,l}) \leq c \frac{|I_{n,l}| \log{|I_{n,l}|}}{2^n}
\end{align*}
and for the conditional expectation we get a lower bound given by the resistance estimate, i.e.
$E_o(\tilde{Z}_{n,l}|Z_{n,l}>0) \geq c'|I_{n,l}|$, hence
\begin{align}
\label{tilde1}
P_o(Z_{n,l}>0) \leq \frac{c}{2^{(1-\beta)n}}
\end{align}
and thus the first sum on the right hand side of \eqref{secondmoment} is upper bounded by $c2^{(2\beta-1)n}$.

For the terms appearing in the second sum on the right hand side of \eqref{secondmoment}
we have $P_o(Z_{n,l}>0,Z_{n,l+k}>0) = P_o(Z_{n,l+k}>0|Z_{n,l}>0) P_o(Z_{n,l}>0)$ and
\begin{align*}
P_o(Z_{n,l+k} &>0|Z_{n,l}>0)\\
&= P_o(Z_{n,l+k}>0, \text{ at least 1 of the RWs hits $J_{n,l+k}$ before $o$ }|Z_{n,l}>0)\\
 &\qquad + P_o(Z_{n,l+k}>0, \text{ both hit $o$ before $J_{n,l+k}$ }|Z_{n,l}>0).
\end{align*}
The first term is bounded by the sum appearing in \eqref{gambler} and hence from the gambler's ruin probability this is upper bounded by $\frac{c}{2^k}$.

And for the second term
\begin{align*}
P_o(Z_{n,l+k}>0, \text{ both hit $o$ before $J_{n,l+k}$ }|Z_{n,l}>0)  &\leq \max_y P_{(o,y)}(Z_{n,l+k}>0) \\
& \leq \max_y \frac{E_{(o,y)}(\tilde{Z}_{n,l+k})}{E_{(o,y)}(\tilde{Z}_{n,l+k}|Z_{n,l+k}>0)}.
\end{align*}
The numerator can be bounded in the same way as we did in \eqref{collisionbd} and the denominator is lower bounded by the effective resistance. So now we get that
\begin{align*}
P_o(Z_{n,l+k}>0, \text{ both hit $o$ before $J_{n,l+k}$ }|Z_{n,l}>0) \leq c 2^{(\beta - 1)n}.
\end{align*}
Hence putting all things together we get
\begin{align*}
P_o(Z_{n,l}>0,Z_{n,l+k}>0) \leq \frac{c}{2^{(1-\beta)n}} \left( \frac{1}{2^k} + \frac{1}{2^{(1-\beta)n}} \right).
\end{align*}

Hence $E_o(Z_n^2) \leq c' 2^{2(2\beta-1)n}$, since $\beta>\frac{1}{2}$.
Thus we have shown that $P_o(Z_n>0) \geq c >0$. Hence by Corollary \ref{cor:fsets}
we obtain $P(Z=\infty) =1$, which completes the proof of Theorem \ref{trees}(c) for $\frac{1}{2} \leq \beta \leq 1$.
\end{proof}

\begin{proof}[\textsl{Proof of Theorem \ref{trees}(d)}]
Let $Z_{n,l}$ count the total number of collisions that happen on the set $J_{n,l}$ and let $\tilde{Z}_{n,l}$ be as in the proof of Claim \ref{notation}, but with the only modification that the time ranges over all $t \in \Z_+$.
We then have
\begin{align*}
P_o(Z_{n,l}>0) \leq \frac{E_o(\tilde{Z}_{n,l})}{E_o(\tilde{Z}_{n,l}|Z_{n,l}>0)}.
\end{align*}
For times $t$ greater than $2(2^{l+1}a)^2$ we get that the expected number of collisions is bounded from above by
$c\frac{|I_{n,l}| \log{|I_{n,l}|}}{2^n}$, which follows  by using the upper bounds for the transition probabilities in $t$ steps.
Here though we are counting the total number of collisions, so we need a better upper bound for the transition probability
for times $t\leq 2(2^{l+1}a)^2$. We are again going to look at the birth and death chain and find the number of collisions
and then divide through by $2^n$.

Let $x$ and $y$ be two points on $\Z_+$ which are at even distance apart and such that $y\leq x$.
Suppose that we start two birth and death chains $X$ from $x$ and $Y$ from $y$ and we couple them in such
a way that $X_t\geq Y_t$ for all $t$ before the first time that they meet and after that time $X_t=Y_t$. From this coupling it follows immediately that
\begin{align*}
p_t^{\text{BD}}(x,0) \leq p_t^{\text{BD}}(y,0).
\end{align*}
If there is no branch point between $x$ and $y$, then we get the same inequality, i.e.
$p_t^{\text{BD}}(0,x) \leq p_t^{\text{BD}}(0,y)$.
If there is one branch point between them, then we get
$p_t^{\text{BD}}(0,x) \leq 2 p_t^{\text{BD}}(0,y)$.

For any $x \in(2^{l}a,2^{l+1}a)$  we have that
\begin{align*}
p_t^{\text{BD}}(0,x) \leq 2 p_t^{\text{BD}}(0,y), \text{ for all } y \in (2^{l-1}a,2^{l}a),
\end{align*}
so $p_t^{\text{BD}}(0,x) \leq \frac{c}{2^{l}a}$, hence
\begin{align*}
E_o(\tilde{Z}_{n,l}) \leq c\frac{|I_{n,l}| \log{|I_{n,l}|}}{2^n} + \frac{1}{2^n}\sum_{t=1}^{2(2^{l+1}a)^2} \frac{1}{(2^{l}a)^2} \leq
c'\frac{|I_{n,l}| \log{|I_{n,l}|}}{2^n}
\end{align*}

Using resistances we get that $E_o(\tilde{Z}_{n,l}|Z_{n,l}>0) \geq c |I_{n,l}|$,
so
\begin{align*}
P_o(Z_{n,l}>0) \leq \frac{c}{2^{(1-\beta)n}}.
\end{align*}
Summing this over all $l=1,\cdots,2^{\beta(n-1)}$ and over all $n$ we get a finite sum, since $\beta < \frac{1}{2}$,
hence by Borel-Cantelli 1 we get that only finitely many of these events occur, so there are only finitely many collisions.
\end{proof}

\section{Concluding Remarks and Questions} \label{sec:open}


\begin{enumerate}

\item
In this paper we have dealt only with collisions of two independent random walks. A natural question to ask is what happens if we have more than two. An easy calculation shows that in $\Z$ the expected number of collisions of three independent random walks is infinite. In fact,
\begin{align*}
E(Z) = E\left(\sum_{t=0}^{\infty} \1(X_t=Y_t=W_t) \right) \ge E\left(\sum_{t=0}^{\infty}\sum_{x: |x|\le \sqrt{t}} \1(X_t=Y_t=W_t=x) \right)\\
 \asymp \sum_{t=0}^{\infty} \sum_{x:|x| \le \sqrt{t}} \frac{1}{(\sqrt{t})^3} = \infty.
\end{align*}
Since $\Z$ is a transitive graph, the number of collisions of the three random walks follows a Geometric distribution. Since the expectation of this number is infinite, it follows that there is an infinite number of collisions with probability 1.

In Comb($\Z, \alpha$) for all $\alpha$, the bounds in Lemma \ref{lem:qtbound} for the transition probabilities imply that the expected number of collisions of three independent random walks is finite.

\item
An application of the infinite collision property of the percolation cluster in $\Z^2$ to a problem in particle systems  is given in \cite{BLZ}.

\item
We have proved that the incipient infinite cluster in high dimensions has the infinite collision property. For the incipient infinite cluster in two dimensions though, the question from \cite{KP} still remains open.

\item
In \cite{Ori} it is proved that the edges crossed by a random walk in a transient network $G$
form a recurrent graph a.s. For which $G$ does the resulting graph have the infinite collision property? This question was asked by Nathana\"el Berestycki.

\item
Let Comb($\Z^2, f$) be a comb with variable lengths over $\Z^2$ defined analogously to Comb($\Z, f$), Definition \ref{defcomb}. For which $f$ does Comb($\Z^2, f$) have the finite collision property? The Green kernel criterion implies that if $f$ has logarithmic growth, then this graph has the infinite collision property.

\item
Suppose that $\{f(n)\}_{n \in \Z}$ are i.i.d. random variables with law $\mu$ supported on $(1,\infty)$. For which $\mu$ does Comb($\Z, f$) have the infinite collision property? This question was raised in \cite{Chen}.
If $\mu$ has finite mean, then $f(n) = o(n)$, so the infinite collision property follows from the Green kernel criterion, Theorem \ref{greenkernel}.


\item
Let $G$ be a graph and let $G'$ be a graph obtained by adding a finite number of vertices and edges. Do $G$ and $G'$ have the same collision property? This question was asked by Zhen-Qing Chen.

\end{enumerate}

\noindent{\Large\bf{Acknowledgements}}

This work was started at the Probability Summer School at Cornell, July 2009, and continued at the workshop on New random geometries in Bath, England. The third author thanks Microsoft Research and the University of Washington, where this work was completed, for their hospitality.


\end{document}